\documentclass[reqno,a4paper]{amsart}

\usepackage{amsfonts,amsmath,amssymb,amsthm}
\usepackage{alltt,esvect,graphicx,stmaryrd,tikz,xcolor}
\usepackage{pgffor}
\usepackage[arrow, matrix, curve]{xy}
\usepackage[titletoc]{appendix}
\usepackage[linesnumbered]{algorithm2e}
\usepackage[T1]{fontenc}
\usepackage[utf8]{inputenc}
\usepackage{helvet}
\usetikzlibrary{calc,intersections,through,shapes,arrows}
\usepackage{mathtools}
\usepackage{subfig}
\usepackage{url}
\usepackage{booktabs}
\usepackage{comment}
\usepackage[textsize=tiny]{todonotes}

\usepackage{enumitem} 
\usetikzlibrary{cd}
\usetikzlibrary{patterns}
\usetikzlibrary{math}
\usetikzlibrary{shapes}
\usetikzlibrary{decorations.fractals}

\usepackage{xstring}

\usepackage{cancel}
\usepackage{hyperref}

\newtheorem{theorem}{Theorem}[section]
\newtheorem{lemma}[theorem]{Lemma}
\newtheorem{proposition}[theorem]{Proposition}
\newtheorem{corollary}[theorem]{Corollary}

\theoremstyle{definition}
\newtheorem{definition}[theorem]{Definition}
\newtheorem{example}[theorem]{Example}

\newtheorem{remark}[theorem]{Remark}

\definecolor{colorX}{RGB}{240,0,0}
\definecolor{colorY}{RGB}{0,0,160}
\definecolor{colorZ}{RGB}{0,160,0}

\newcommand{\N}{\mathbb{N}}
\newcommand{\Z}{\mathbb{Z}}

\newcommand{\R}{\mathbb{R}}
\newcommand{\C}{\mathbb{C}}

\newcommand{\A}{\mathbb{A}}
\newcommand{\T}{\mathbb{T}}
\newcommand{\V}{\mathbb V}

\newcommand{\CF}{\mathcal{F}}
\newcommand{\CK}{\mathcal{K}}

\newcommand{\CM}{{\mathcal M}}

\newcommand{\CE}{\mathcal{E}}
\newcommand{\CH}{\mathcal{H}}
\newcommand{\CO}{\mathcal{O}}

\newcommand{\Hom}{\mathop{\rm Hom}\nolimits}

\newcommand{\Id}{\mathop{\rm Id}\nolimits}

\renewcommand{\ker}{\mathop{\rm ker}\nolimits}
\newcommand{\im}{\mathop{\rm im}\nolimits}

\newcommand{\ord}{\operatorname{ord}}

\newcommand{\Quot}{\mr{Quot}\,}
\newcommand{\Spec}{\mr{Spec}\,}

\newcommand{\mr}{\mathrm}

\newcommand{\mc}{\mathcal}

\newcommand{\bp}{\begin{para}}
\newcommand{\ep}{\end{para}}
\newcommand{\bps}{\begin{paras}}
\newcommand{\eps}{\end{paras}}
\newcommand{\benum}{\begin{enumerate}[{\rm(i)}]}
\newcommand{\eenum}{\end{enumerate}}

\newcommand{\Div}{\operatorname{Div}}
\newcommand{\Cl}{\operatorname{Cl}}

\newcommand{\kSt}{\,\big|\;}

\newcommand{\surj}{\rightarrow\hspace{-0.8em}\rightarrow}

\newcommand{\kss}{\scriptscriptstyle}
\newcommand{\kbb}{{\kss \bullet}}

\newcommand{\toric}{\T\V}  
\newcommand{\opn}{\operatorname}

\newcommand{\kG}{\Gamma}

\newcommand{\ssect}[1]{Subsection~\ref{#1}}

\newcommand{\CT}{{\mathcal T}}

\newcommand{\Matr}[2]{\left(\begin{array}{@{}*{#1}{r}@{}} #2
\end{array}\right)}

\DeclareMathOperator{\spann}{span}
\DeclareMathOperator{\diag}{diag}

\definecolor{oliwkowy}{HTML}{627037}
\definecolor{lightblue}{RGB}{135,206,250}
\definecolor{darkblue}{RGB}{0,0,160}
\definecolor{darkgreen}{RGB}{0,160,0}
\definecolor{veryPeri}{RGB}{102,103,171}
\definecolor{intOrange}{rgb}{1.0,.310,.0}
\definecolor{MidnightBlue}{RGB}{102,103,171}

\newcounter{para}[section]
\newenvironment*{para}[1]{\refstepcounter{para}\noindent\ignorespaces{\bf\thepara.~#1.}}{\ignorespacesafterend\bigskip}
\newenvironment*{paras}[1]{{\bf #1.}}{\ignorespacesafterend\bigskip}
\numberwithin{para}{section}

\newcommand{\PP}{\mathbb P}

\setlist[enumerate,1]{label = (\roman*),ref = \theenumii.\roman*}

\newcommand{\kk}{{\operatorname{k}}}

\renewcommand{\div}{\operatorname{div}}\newif\ifnotext

\definecolor{colorAB}{RGB}{240,0,0}
\definecolor{colorBC}{RGB}{0,0,160}
\definecolor{colorCD}{rgb}{1.0,.310,.0}
\definecolor{colorDA}{RGB}{0,160,0}

\newcommand{\TT}{{\opn{T}}}

\definecolor{cRayA}{rgb}{0.0,0.6,0.0}
\definecolor{cRayB}{rgb}{0.0,0.0,1.0}
\definecolor{cRayC}{rgb}{1.0,0.0,0.0}

\definecolor{cConeA}{rgb}{0.0,0.6,0.0}
\definecolor{cConeB}{rgb}{0.0,0.0,1.0}
\definecolor{cConeC}{rgb}{1.0,0.0,0.0}

\newcommand{\RSK}{\mathcal{K}_{\CE}}
\definecolor{green}{rgb}{0.0,0.6,0.0}
\definecolor{blue}{rgb}{0.0,0.0,1.0}
\definecolor{red}{rgb}{1.0,0.0,0.0}

\mathchardef\mhyphen="2D

\newcommand{\OD}{\mathcal O_X(\WD)}
\newcommand{\CV}{\mathcal V}
\newcommand{\CU}{\mathcal U}

\newcommand{\idm}{\mathfrak{m}}
\newcommand{\hatDiv}{\operatorname{\Div}}

\newcommand{\Cox}{\operatorname{Cox}}
\newcommand{\QCox}{{\Cox}}
\newcommand{\ux}{\underline{x}}
\newcommand{\uz}{\underline{z}}

\newcommand{\PD}{{P}}
\newcommand{\slice}{{slice}}

\newcommand{\HM}{{\CH\hspace{-0.09em}\CM}}
\newcommand{\partialP}{{\partial}}

\newcommand{\dd}{{\scalebox{1.1}{$\opn{d}$}}}
\newcommand{\WD}{\mathcal{W}}
\newcommand{\agnate}{{agnate}}
\newcommand{\KS}{{K}}
\newcommand{\MN}[2]{{\langle#1,#2\rangle}}
\DeclareMathOperator{\antidiag}{antidiag}
\newcommand{\rhoSig}{\hat{\rho}_\sigma} 
\newcommand{\piN}{{\pi}}
\newcommand{\ta}{{\widetilde{a}}}
\newcommand{\taSig}{{a(\sigma)}}
\newcommand{\rhoSt}{{\rho^*}}
\newcommand{\rhoStSig}{\hat{\rho}^*_\sigma}
\newcommand{\oldMinus}{}
\newcommand{\oldMinusWirdPlus}{{+}}
\newcommand{\oldPlusWirdMinus}{{-}}
\newcommand{\newMinus}{{-}}
\newcommand{\sliceU}{\CU}

\begin{document}
\parindent0mm

\title[The Weil decoration of the Horrocks-Mumford bundle]
{The Weil decoration of the Horrocks-Mumford bundle}

\author[K.~Altmann]{Klaus Altmann%
}
\address{Institut f\"ur Mathematik,
FU Berlin,
K\"onigin-Luise-Str.~24-26,
D-14195 Berlin
}
\email{altmann@math.fu-berlin.de}
\author[A.~Hochenegger]{Andreas Hochenegger}
\address{
Dipartimento di Matematica ``Francesco Brioschi'',
Politecnico di Milano,
via Bonardi 9,
20133 Milano 
}
\email{andreas.hochenegger@polimi.it}
\author[F.~Witt]{Frederik Witt}
\address{
Fachbereich Mathematik,
U Stuttgart,
Pfaffenwaldring 57,
D-70569 Stuttgart
}
\email{witt@mathematik.uni-stuttgart.de}
\thanks{ 
\hfill\newline
}

\begin{abstract}
For a normal algebraic variety we generalise the relation between reflexive rank one sheaves and Weil divisors to reflexive sheaves of arbitrary rank and so-called Weil decorations. As an application, we define and study a natural generalisation of the celebrated Horrocks-Mumford bundle. 
\end{abstract}

\maketitle

\section{Introduction}
\label{sec:Intro}
Let $X$ be a normal algebraic variety. A coherent sheaf on $X$ is said to be {\em reflexive} if the natural inclusion into its double dual is actually an isomorphism; in particular, it is torsion-free. Reflexive sheaves define a handier and more versatile class than locally free sheaves, see~\cite{hartreflexiv}. For instance, every Weil divisor $D$ of $X$ specifies a reflexive sheaf of rank one $\CO_X(D)$ inside the field of rational functions $K:=K(X)$ of $X$: If $U\subseteq X$ is open, then  
\begin{equation}
\label{eq:DefOXD}
\CO_X(D)(U):=\{f\in K^*\mid\big(D+\div(f)\big)|_U\geq0\}\cup\{0\}.
\end{equation}
Conversely, any reflexive rank one sheaf is isomorphic to some $\CO_X(D)$. 

\medskip

We generalise this correspondence to reflexive sheaves of higher rank. Namely, any reflexive sheaf $\CE$ sits inside its generic stalk $\CE_\eta$ by torsion-freeness, and every $0\neq e\in\CE_\eta$ gives rise to the reflexive rank one sheaf 
\[
\CE(e)(U):=\big(K\cdot e\big)\cap\CE(U)\stackrel{1/e}{\hookrightarrow}\KS.
\]
Therefore, we can associate with $e$ a unique Weil divisor $D(e)$. We call
\[
\WD_\CE\colon\CE_\eta\setminus\{0\}\to\Div(X),\quad e\mapsto D(e).
\]
the {\em Weil decoration of $\CE$}; it behaves like a $K$-valuation on $\CE_\eta$, cf.\ Proposition~\ref{prop:WDRS}. Moreover, any such assignment $\CV\setminus{0}\to\Div(X)$ on a finite dimensional $K$-vector space $\CV$ arises this way (Proposition~\ref{prop:Bijection}). Finally, sheaf morphisms $\varphi\colon\CE\to\CF$ between reflexive sheaves translate into $K$-linear maps $\varphi_\eta\colon\CE_\eta\to\CF_\eta$ with $\WD_\CE(e)\leq\WD_\CF(\varphi_\eta(e))$ which we take as morphisms between Weil decorations. 

\medskip

{\bf Theorem A (see~\ref{thm:EquivalenceCat}).} 
{\em The category of reflexive sheaves is equivalent to the category of Weil decorations.}

\medskip

{\bf Remark.} The idea of a Weil decorations goes back to a previous construction of the authors in the context of toric geometry~\cite{TRS}. In fact, both constructions are equivalent for toric sheaves, that is, torus linearised reflexive sheaves on a toric variety, see \ssect{subsec:TS}. 

\medskip

A major thread of this article is to supply tools for the computation of Weil decorations. Let $\WD(e)_\PD$ denote the coefficient of $\WD(e)$ at the prime divisor $\PD$ of $X$. 

\medskip

{\bf Theorem B (see~\ref{prop:KerWeilDecos} and~\ref{prop:QuotWeilDecos}).} 
{\em Let $\CE$ and $\CE'$ be two reflexive sheaves with Weil decorations $\WD$ and $\WD'$. 

\smallskip

{\rm(i)} If $\CE'$ is the kernel of a morphism $\CE\to\CE''$ with $\CE''$ torsion-free, then $\WD'=\WD|_{\CE'_\eta}$.

\smallskip

{\rm(ii)} If $\mu\colon\CE\to\CE'$ is surjective, then $\WD'(e')_\PD=\max\limits_{e\in\mu_\eta^{-1}(e')}\WD(e)_\PD$.}

\medskip

Theorem B is tailor-made for the computation of Weil decorations of sheaves given by a {\em monad}. Indeed, passing to the generic stalks turns this into linear algebra combined with an optimisation problem to determine the maximum. As an example, we compute the Weil decoration of the celebrated Horrocks-Mumford bundle on $\PP^4$~\cite{hm}. Its importance stems from the fact that it is so far the only known indecomposable rank two vector bundle on $\PP^4$ in characteristic $0$. 

\medskip

Turning the tables we can also make use of Weil decorations to define reflexive sheaves. Consider $X$ together with a simple normal crossing divisor $D$. Moreover, choose a unit $h_\PD$ in the residue field $\kappa(\PD)$ for any prime divisor $\PD$ supporting $D$; for $f\in K$, we let $f(\PD)$ be the value of $f$ at $\PD$ if defined.

\medskip

{\bf Theorem C (see~\ref{prop:SemiNorm} and~\ref{prop:WDHMtype}).} 
{\em The map $\WD\colon\!K(X)^2\!\setminus\!\{(0,0)\}\!\to\!\Div(X)$ given by 
\[
\WD(f,g)_\PD=
\begin{cases}
\min\{\ord_P(f),\ord_P(g)\}+1,&\PD\in\opn{supp}(D)\text{ and }(f/g)(\PD)=h_\PD\\
\min\{\ord_P(f),\ord_P(g)\},&\text{else,}
\end{cases}
\]
defines a Weil decoration.}

\medskip

For instance, we recover the Horrocks-Mumford bundle $\HM$ on $\PP^4$ by taking $D=\partial\PP^4=\sum_{\rho=0}^4H_\rho$ for the coordinate hyperplanes $H_\rho=\{z_\rho=0\}$, and
\begin{equation}
\label{eq:Hrho}
h_{H_\rho}=z_{\rho+1}z_{\rho-1}/z_{\rho+2}z_{\rho-2}\in\kappa^*(H_\rho),\quad\rho\in\Z/5\Z.
\end{equation}
In fact, the assignment in~\eqref{eq:Hrho} readily generalises to toric varieties $X$ with canonical divisor $D=\partial X$, and we consider some examples in Section~\ref{sec:HMtype}.

\bigskip

{\em Conventions.} In this article we let $\kk$ be an algebraically closed field of characteristic zero. We always work with normal algebraic $\kk$-varieties, that is, normal, separated and integral schemes of finite type over $\kk$. In particular, $X$ is regular in codimension one: any one-dimensional local ring is regular, thus a discrete valuation ring (DVR). 

\medskip

We let $\eta$ be the generic point of $X$ and $K=K(X)$ be the field of rational functions as well as the induced constant sheaf. By convention, $\PD$ denotes a prime divisor and its generic point, thus both notations $\PD\subseteq X$ and $\PD\in X$ will be used. The group of Weil divisors on $X$ will be written $\Div(X)$. Finally, if $\PD\in X$ is a prime divisor with residue field $\kappa(\PD)=\CO_{X,\PD}/\idm_{X,\PD}$, then evaluation of $f\in K$ at $P$ gives an element $f(\PD)\in\kappa(\PD)\cup\{\infty\}$; in particular, $f(\PD)$ is finite if and only if $f\in\CO_{X,\PD}$.

\section{Weil decorations}
\label{sec:WeilDeco}
Let $\PD$ be a prime divisor with associated valuation $\ord_\PD\colon K\to\Z$ and discrete valuation ring (DVR) $\CO_{X,\PD}$; we gloss over the usual convention of assigning to $0$ the formal value $\infty$. On $\Div(X)$ consider the poset structure 
\[
D\geq D'\quad\Longleftrightarrow\quad D-D'\geq0,\text{ that is, }D-D'\text{ is effective.}  
\]
The greatest lower bound or {\em meet} of two divisors is given by
\[
D\wedge D':=\min\{D,D'\}:=\sum\min\{D_\PD,D'_\PD\}\cdot\PD, 
\]
where $D_\PD$ is the coefficient of $D\in\Div(X)$ with respect to the prime divisor $\PD$. Similarly, we define the smallest upper bound or {\em join} of two divisors by
\[
D\vee D':=\max\{D,D'\}:=\sum\max\{D_\PD,D'_\PD\}\cdot\PD. 
\]

\subsection{Pre-Weil decorations}
\label{subsec:PWD}
Let $\CV$ be an $r$-dimensional $K$-vector space.

\begin{definition}
\label{def:PreWD}
A {\em pre-Weil decoration on $\CV$} is an assignment $\WD\colon\CV\to\hatDiv(X)$ satisfying
\begin{itemize}
\item[{\rm (W0)}] $\WD(v)=\infty$ if and only if $v=0$;

\smallskip

\item[{\rm (W1)}] for all $f\in K$ and $v\in\CV$, we have $\WD(f\cdot v)=\div(f)+\WD(v)$;

\smallskip

\item[{\rm (W2)}] for all $v$, $v'\in\CV$, we have $\WD(v+v')\geq\WD(v)\wedge\WD(v')$.
\end{itemize}
The {\em rank} of the pre-Weil decoration is $r=\dim_K\CV$.
\end{definition}

\begin{remark}
A pre-Weil decoration induces on $\CV$ a family of non-archimedian semi-norms over the valued fields $(K,\ord_\PD)$ given by the $\PD$-coefficients of $\WD$, namely
\[
\CV\to\Z,\quad v\mapsto|v|_\PD:=\WD(v)_\PD.
\]
Conversely, any such a $\PD$-indexed family gives a pre-Weil decoration defined by 
\begin{equation}
\label{eq:WDFam}
\WD(v)=\sum|v|_\PD\PD 
\end{equation}
provided that $|v|_\PD=0$ except for finitely many prime divisors.
\end{remark}

The geometric relevance of pre-Weil decorations is this.

\begin{proposition}
\label{prop:QCSPWD}
Let $\WD\colon\CV\to\hatDiv(X)$ be a pre-Weil decoration. Then
\[
\OD(U):=\{v\in\CV\mid\WD(v)|_U\geq0\}\subseteq\CV
\]
defines the quasi-coherent sheaf {\em $\OD$ associated with $\WD$}. Its generic stalk is $\CV$.
\end{proposition}

\begin{example}
\label{exam:LineBundle}
In view of (W1) in Definition~\ref{def:PreWD}\,, a pre-Weil decoration $\WD\colon K\to\hatDiv(X)$ is already determined by $D=\WD(1)$ in $\Div(X)$. In particular, its associated sheaf $\OD$ is precisely the sheaf $\CO_X(D)$ from~\eqref{eq:DefOXD}.
\end{example}

\begin{proof}[Proof of Proposition~\ref{prop:QCSPWD}]
Since for two open subsets $U'\subseteq U$, the corresponding restriction map is just inclusion, $\OD$ is indeed a sheaf. Moreover, $\div(f)|_U\geq0$ for any function $f\in K$ regular on $U$. Hence (W0), (W1) and (W2) immediately imply that $\OD(U)$ is an $\CO_X(U)$-module.

\smallskip

To check quasi-coherency let $U=\Spec A$ be open and $f\in A$. Then $s\in\OD(U_f)$ implies $\WD(s)_\PD\geq0$ for all $\PD\in U_f$ while $\ord_\PD(f)>0$ on $\PD\in U\setminus U_f$. Hence $f^Ns\in\OD(U)$ for a suitable $N\in\N$ which implies $\OD(U_f)=\OD(U)_f$. 

\smallskip

Finally, $\OD_\eta=\CV$ follows from the inclusion $\OD(U)\subseteq\CV$.
\end{proof}

\begin{remark}
(i) For all $\PD\in X$, $\,\OD_\PD=\{v\in\CV\mid\WD(v)_\PD\geq0\}$.

\smallskip

(ii) For all $f\in K^*$ we have $\CO_X(\WD+\div(f))=f^{-1}\cdot\OD$. 

\smallskip

(iii) For all $v\in V$ we have  
\begin{equation}
\label{eq:OD}
\CO_X(\WD(v))\cdot v=\KS\cdot v\cap\OD.
\end{equation}
Indeed, $f\cdot v\in\big(\KS\cdot v\cap\OD\big)(U)$ for $f\in K$ if and only if $f\cdot v\in\OD(U)$. By (W1), this is equivalent to $f\in\CO_X(\WD(v))(U)$ and thus to $f\cdot v\in\CO_X(\WD(v))(U)\cdot v$.
\end{remark}

The following semi-norms serve as building blocks for pre-Weil decorations we consider in this article, cf.\ Example~\ref{exam:PreWeilNotWeil} and Definition~\ref{def:HMtype}.

\begin{proposition}
\label{prop:SemiNorm}
Let $\PD$ be a prime divisor of $X$ and $h_\PD\in\kappa(\PD)^*$. Then 
\[
\varphi_{h,\PD}(f,g)=
\begin{cases}
\min\{\ord_\PD(f),\ord_\PD(g)\}+1,&\tfrac fg(\PD)=h_\PD\\
\min\{\ord_\PD(f),\ord_\PD(g)\},&\text{else}
\end{cases}
\]
induces a non-archimedian semi-norm on $K^2$ over $(K,\ord_\PD)$.
\end{proposition}

\begin{remark}
The first case implies $\ord_\PD(f)=\ord_\PD(g)$.  
\end{remark}

\begin{proof}[Proof of Proposition~\ref{prop:SemiNorm}]
By definition, $\varphi_{h,\PD}(\lambda\cdot(f,g))=\ord_\PD(\lambda)+\varphi_{h,\PD}(f,g)$ for any $\lambda\in K^*$. To check the strong triangle inequality we must show that for any $v=(f,g)$ and $v'=(f',g')$ in $K^2\setminus\{(0,0)\}$, the inequality
\begin{equation}
\label{eq:DesIneq}
\varphi_{h,\PD}(v+v')\geq\min\{\varphi_{h,\PD}(v),\,\varphi_{h,\PD}(v')\}
\end{equation}
holds. To lighten notation we set $|\cdot|:=\ord_P(\cdot)$. Then by definition,
\[
\varphi_{h,\PD}(v)=\min\{|f|,|g|\}+\epsilon(v),\quad\epsilon(v)\in\{0,1\}.  
\]
The problematic case is therefore \fbox{$\epsilon(v+v')=0$} while $\epsilon(v)$ or $\epsilon(v')$ is nontrivial, say $\epsilon(v)=1$. In particular, $|f|=|g|$. We may assume the equalities
\[
\varphi_{h,\PD}(v+v')=\min\{|f+f'|,|g+g'|\}=|f+f'|=\min\{|f|,|f'|\};
\]
the first by symmetry and the second since otherwise, \eqref{eq:DesIneq} holds trivially.

\medskip

{\em Case 1:} $\epsilon(v')=0$. This entails 
\[
\min\{\varphi_{h,\PD}(v),\,\varphi_{h,\PD}(v')\}=\min\{|f|+1,|g|+1,|f'|,|g'|\}=\min\{|f|+1,|f'|,|g'|\}. 
\]
Assuming~\eqref{eq:DesIneq} not to hold implies $|f+f'|<\min\{|f|+1,|f'|,|g'|\}$ and thus
\[
\min\{|f|,|f'|\}=|f+f'|<|f|+1,\,|f'|,\,|g'|.
\]
Hence $|f+f'|=|f|<|f'|$, $|g'|$ and so $|g|=|f|<|g'|$. Therefore $\tfrac{f+f'}{g+g'}(\PD)=\tfrac{f}{g}(\PD)=h_\PD$, which leads to $\epsilon(v+v')=1$, contradicting our initial assumption. 

\medskip

{\em Case 2:} $\epsilon(v')=1$. Then $\tfrac{f'}{g'}(\PD)=h_\PD\quad\text{and}\quad|f'|=|g'|$. Now we cannot have $|g'|=|f'|=|g|=|f|$ for $f\equiv h_\PD\cdot g+\opn{mod}\idm_{X,\PD}$ and $f'\equiv h_\PD\cdot g'+\opn{mod}\idm_{X,\PD}$ implies $f+f'\equiv h_\PD\cdot(g+g')+\opn{mod}\idm_{X,\PD}$, contradicting $\epsilon(v+v')=0$. On the other hand, if, say, $|g'|=|f'|>|g|=|f|$, then $\tfrac{f+f'}{g+g'}(\PD)=\tfrac{f}{g}(\PD)=h_\PD$ contradicts again $\epsilon(v+v')=1$.
\end{proof}

\subsection{Weil decorations and their reflexive sheaf}
\label{subsec:RSWD}

The sheaf $\CO_X(D)$ associated with the pre-Weil decoration $\WD\colon K\to\hatDiv(X)$ sending $1$ to $D$ (cf.\ Example~\ref{exam:LineBundle}) is actually coherent. As we will see in a moment, this is not necessarily true for general pre-Weil decorations. We therefore make the following

\begin{definition}
A a pre-Weil decoration $\WD\colon\CV\to\hatDiv(X)$ is {\em coherent}, if its associated sheaf $\OD$ is coherent. A {\em Weil decoration} is a coherent pre-Weil decoration. 
\end{definition}

For a practical coherence criterion we borrow terminology from the theory of Banach spaces over non-archimidean fields~\cite{monna}.

\begin{definition}
\label{def:Orthogonal}
Let $\WD\colon\CV\to\hatDiv(X)$ be a pre-Weil decoration of rank $r$. A set of vectors $v_1,\ldots,v_s\in\CV$ is called $\PD$-{\em orthogonal} for some prime divisor $\PD$ of $X$, if for all $f_1,\ldots,f_s\in K$,
\[
\WD\big(\sum_{i=1}^sf_iv_i\big)_\PD=\min\{\div(f_i)_\PD\mid i=1,\ldots,s\},
\]
and $U$-{\em orthogonal} if this holds simultanously for all prime divisors inside some open subset $U$ of $X$. We call $\WD$ {\em trivial} if $\WD$ admits an $X$-orthogonal set with $s=r$.
\end{definition}

\begin{remark}
\label{rem:LinInd}
Property (W0) implies that any set of of $\PD$- or $U$-orthogonal vectors must be $K$-linearly independent. If possible, the choice of an $X$-orthogonal basis of $\CV$ induces an isomorphism $\OD\cong\CO_X^r$ (the converse follows directly from Theorem~\ref{thm:EquivalenceCat}). In particular, any trivial pre-Weil decoration is a Weil decoration.
\end{remark}

\begin{definition}
Two pre-Weil decorations $\WD$, $\WD'$ on $\CV$ are {\em \agnate}, if there exists a divisor $D\in\Div(X)$ such that 
\begin{equation}
\label{eq:AgnateInequal}
\WD-D\leq\WD'\leq\WD+D.
\end{equation}
Equivalently, there exists an open set $U$ of $X$ such that $\WD|_U=\WD'|_U$.
\end{definition}

\begin{remark}
We can replace $D$ in~\eqref{eq:AgnateInequal} by any $D'\geq D$. In particular, over an open affine $X=\Spec A$ we can take the divisor of a suitable regular function $f$ for $D'$. Therefore, $\WD'$ and $\WD$ are agnate if and only if we can find an $f\in A$ with
\begin{equation}
\label{eq:Agnate}
f\cdot\OD=\CO_X(\mc W-\div(f))\subseteq\CO_X(\WD')\subseteq\CO_X(\mc W+\div(f))=f^{-1}\OD 
\end{equation}
inside $\CV$.
\end{remark}

\begin{proposition}
\label{prop:CohCond}
The Weil decoration $\WD$ is coherent if and only if $\WD$ is \agnate\ to a trivial one.
\end{proposition}

\begin{proof}
For the implication we show that any two Weil decorations $\WD$ and $\WD'\colon\CV\to\hatDiv(X)$ are agnate. This is a local condition, so we may assume $X=\Spec A$. Now $\CO_X(\WD)$, $\CO_X(\WD')\subseteq\CV$ are finitely generated $A$-modules with
\[
\CO_X(\WD)\otimes_A\Quot(A)=\CV=\CO_X(\WD')\otimes_A\Quot(A).  
\]
In particular, there exists an $f\in A$ such that~\eqref{eq:Agnate} holds, that is, $\CO_X(\WD)$ and $\CO_X(\WD')$ are agnate. The converse is clear.  
\end{proof}

\begin{example}
\label{exam:PreWeilNotWeil}
To construct a pre-Weil decoration which is not coherent we consider the semi-norms $\varphi_{h,\PD}$ from Proposition~\ref{prop:SemiNorm}. For instance, let $X=\C^1$ and 
\[
h(\PD):=\exp(-\PD)\in\kappa(\PD)^*=\C^* 
\]
for any closed point $\PD\in\C$. By non-rationality of $h$, the equality $(f/g)(\PD)=h(\PD)$ can hold only for finitely many prime divisors. Hence, our family of semi-norms induces a pre-Weil decoration, which, however, is not agnate to a Weil decoration.
\end{example}

The $\CO_X$-module $\CE:=\OD$ is torsionfree since it sits inside $\CV$. By~\cite[1.6]{hartreflexiv}, a torsion-free sheaf is reflexive if and only if the restriction maps $\CE(U)\to\CE(U\setminus Y)$ are bijective for any closed subset $Y$ of codimension two or higher; this holds in our situation, because $U$ and $U\setminus Y$ contain the same prime divisors. Therefore:

\begin{proposition}
\label{prop:RSWD}
The sheaf $\OD$ of a Weil decoration $\WD\colon\CV\to\hatDiv(X)$ is reflexive. 
\end{proposition}

\subsection{The Weil decoration of a reflexive sheaf}
\label{subsec:WDRS}

So far, we associated with a Weil decoration $\WD$ a reflexive sheaf $\OD$; the converse will occupy us next. Let $\CE$ be a reflexive sheaf of rank $r$ over $X$ with generic stalk $\CE_\eta=\varinjlim_{U\not=\varnothing}\CE(U)\cong K^r$.
Since $\CE$ is torsion-free, we will always consider $\CE$ as an $\CO_X$-subsheaf of the constant sheaf induced by $\CE_\eta$. For $0\neq e\in\CE_\eta$ we define the rank one sheaf $\CE(e)$ by
\[
\fbox{$\CE(e)(U):=\big(K\cdot e\big)\cap\CE(U)\subseteq\CE_\eta$}
\]
on $U\subseteq X$ open. Since $\CE$ is reflexive and $\CE(e)$ is saturated, $\CE(e)$ is actually reflexive~\cite[II.1.1.16]{oss}, too. The isomorphic subsheaf
\[
\RSK(e):=\frac 1e\cdot\CE(e) 
\]
of $\KS$ resulting via
\[
\begin{tikzcd}
\CE(e)\ar[r,hookrightarrow]\ar[d,"\cdot 1/e"',"\cong"]&\KS\cdot e\ar[d,"\cdot 1/e","\cong"']\\
\RSK(e)\ar[r,hookrightarrow]&\KS
\end{tikzcd}
\]
induces a well-defined Weil divisor $D(e)$ with $\CO_X(D(e))=\RSK(e)$ and thus the map
\[
\WD_\CE\colon\CE_\eta\to\hatDiv(X),\quad 0\neq e\mapsto\WD_\CE(e):=D(e).
\]
Differently put, $e\neq0$ gives
\begin{equation}
\label{eq:WDRS}
\CO_X(\WD_\CE(e))\cdot e=\CK_\CE(e)\cdot e=\CE(e)=(\KS\cdot e)\cap\CE.
\end{equation}

\begin{proposition}
\label{prop:WDRS}
Let $e$, $e'\in\CE_\eta$ and $f\in K$. Then 
\begin{itemize}
\item[{\rm (i)}] $\WD_\CE(f\cdot e)=\div(f)+\WD_\CE(e)$;

\smallskip

\item[{\rm (ii)}] $\WD_\CE(e+e')\geq\WD_\CE(e)\wedge\WD_\CE(e')$;

\smallskip

\item[{\rm (iii)}] $\CO_X(\WD_\CE)=\CE$. 
\end{itemize}
In particular, $\WD_\CE$ defines a Weil decoration of rank $r$.
\end{proposition}

\begin{proof}
This follows directly from the construction. For instance, \eqref{eq:WDRS} and \eqref{eq:OD}, namely $\CO_X(\WD(v))\cdot v=\KS\cdot v\cap\OD$, imply
\[
(\KS\cdot e)\cap\CE(U)=(\KS\cdot e)\cap\CO_X(\WD_\CE)(U)
\]
for all $e\neq0$ whence $\CE=\CO_X(\WD_\CE)$.
\end{proof}

\begin{proposition}
\label{prop:Bijection}
The maps $\sigma\colon\WD\mapsto\OD$ and $\tau\colon\CE\mapsto\WD_\CE$, which are defined on Weil decorations and reflexive sheaves on $X$, respectively, are mutually inverse. 
\end{proposition}

\begin{proof}
By Propositions~\ref{prop:RSWD} and~\ref{prop:WDRS}\,, the maps are well-defined. Furthermore, item (iii) of Proposition~\ref{prop:WDRS} implies that $\sigma\circ\tau$ is the identity on reflexive sheaves. It remains to prove $\WD=\WD_{\OD}$. By design,
\[
\WD_{\OD}=\tau(\OD)=\tau\circ\sigma(\WD).  
\]
Applying $\sigma$ yields $\sigma(\WD_{\OD})=\sigma(\WD)$, and we are left with showing injectivity of $\sigma$. Now $\OD=\CO_X(\WD')$ entails that $\WD(v)_\PD\geq0$ if and only if $\WD'(v)_\PD\geq0$. Replacing $v$ by a suitable $f\cdot v$, $f\in K$, yields $\WD=\WD'$.
\end{proof}

\section{Slices}
\label{sec:SliWeilDec}

\subsection{\texorpdfstring{$\PD$}{P}-orthogonal bases}

Let $\CE$ be a reflexive sheaf of rank $r$ on $X$. Since $\CO_{X,\PD}$ is a DVR for any prime divisor $\PD$, the module $\CE_\PD$ is free of rank $r$ .

\begin{proposition}
\label{prop:CompDe}
A set of vectors $\{e_1,\ldots,e_r\}$ in $\CE_\eta$ is $\PD$-orthogonal for $\WD_\CE$, cf.\ Definition~\ref{def:Orthogonal}, if and only if it defines an $\CO_{X,\PD}$-basis of $\CE_\PD$. 
\end{proposition}

\begin{proof}
For the implication, we note that a $\PD$-orthogonal set is $K$- and thus $\CO_{X,\PD}$-linearly independent by Remark~\ref{rem:LinInd}. Since $\{e_1,\ldots,e_r\}$ defines a $K$-basis of $\CE_\eta$, every $e\in\CE_\PD\subseteq\CE_\eta$ can be written as $e=\sum_{i=1}^nf_ie_i$ for $f_i\in K$. Then $\PD$-orthogonality implies $\min_i\{\ord_\PD(f_i)\}=\WD_\CE(e)_\PD\geq0$ whence $f_i\in\CO_{X,\PD}$. Moreover, $\WD_\CE(e_i)_\PD=0$ so that $e_1,\ldots,e_n\in\CE_\PD$. 

\smallskip

Conversely, pick $e=\sum_{i=1}^rf_ie_i\in\CE_\eta$ and $f\in K^*$. By design, $f\cdot e$ is in $\CE_\PD$ if and only if $f\cdot e$ is in $(\KS\cdot e)\cap\CE_\PD=\CO_X(\WD_\CE(e))_\PD\cdot e$. Further, $\CE_P=\bigoplus_{i=1}^r\CO_{X,P}\,e_i$ entails 
\begin{align*}
\CO_X(\WD_\CE(e))_\PD&=\{f\in K\mid f\cdot e\in\CE_\PD\}\\
&=\{f\in K\mid f\cdot f_i\in\CO_{X,\PD},\,i=1,\ldots,r\}=\bigcap_{i=1}^r\;f_i^{-1}\cdot\CO_{X,\PD}\subseteq K.
\end{align*}
Now fix a local parameter $t$ of $\CO_{X,\PD}$ and write $(f_i)=(t^{\ord_\PD(f_i)})$ for the fractional $\CO_{X,\PD}$-ideal in $K$ generated by $f_i$. Then
\[
\bigcap_{i=1}^r \;f_i^{-1}\cdot\CO_{X,\PD}=
t^{\max\{-\ord_\PD(f_i)\,\mid\, i=1,\ldots,r\}}\cdot\CO_{X,\PD}
=t^{-\min\{\ord_\PD(f_i)\,\mid\, i=1,\ldots,r\}}\cdot\CO_{X,\PD},
\]
that is, $\WD_\CE(e)_\PD=\min\{\ord_\PD(f_i)\mid i=1,\ldots,r\}$.
\end{proof}

\begin{example}
\label{exam:WDOmega}
Let $X$ be a curve. Any local parameter $t\in K(X)$ defines a $K$-basis $ \dd t$ of the rational $1$-forms $\Omega_{K/\kk}=\Omega_\eta$. If $\alpha=f\dd t$, then 
\[
\WD_{\Omega_X}(\alpha)_\PD=\ord_\PD(f)
\]
whenever $t$ defines a local parameter at $\PD$. In particular, $\WD_{\Omega_X}(\alpha)=\div(\alpha)$, the usual divisor of a rational one-form.
\end{example}

\begin{remark}
A $K$-basis $B=\{e_1,\ldots,e_r\}$ of $\CE_\eta$ induces an $\CO_{\sliceU_B}$-basis of $\CE|_{\sliceU_B}$ for a maximal nonempty open set $\sliceU_B\subseteq X$. Then $B$ is a $\PD$-orthogonal basis for $\WD_\CE$ if and only if $\PD\in \sliceU_B$. In particular, $B$ is $\sliceU_B$-orthogonal. 
\end{remark}

\begin{definition}
A {\em \slice\ $E$ of $\CE$} is $\kk$-vector space in $\CE_\eta$ such that $E\otimes_\kk K=\CE_\eta$. We call the restriction
\[
\WD_E\colon E\to\hatDiv(X),\quad\WD_E(e):=\WD_\CE(e)
\]
the {\em $E$-slice of $\WD_\CE$}. 
\end{definition}

\begin{remark}
\label{rem:Slices}
(i) Any $K$-basis $B$ of $\CE_\eta$ generates a \slice\ $E$ over $\kk$; conversely, any $\kk$-basis of a \slice\ $E$ yields a $K$-basis of $\CE_\eta$. Any two $\kk$-bases $B$, $B'$ in a given slice $E$ satisfy $\sliceU_B=\sliceU_{B'}$ and we therefore write $\sliceU_E$ instead of $\sliceU_B$ and $\sliceU_{B'}$. In particular,
\[
E\otimes_\kk\CO_{\sliceU_E}=\CE|_{\sliceU_E}\quad\text{and}\quad E\otimes_\kk\CO_{X,\PD}=\CE_\PD 
\]
for any $\PD\in\sliceU_E$. Furthermore, $\WD_E|_{\sliceU_E}\equiv0$, that is, the $E$-\slice\ of $\WD_\CE$ is {\em finitely supported} in the sense that $\WD_E$ takes values in 
\[
\Div(E):=\langle\PD_1,\ldots,\PD_m\rangle=\Z^m 
\]
for the finitely many prime divisors $\PD_i\in X\setminus\sliceU_E$.

\smallskip

(ii) As the reflexive sheaves of $\HM$-type to be discussed in Section~\ref{sec:HMtype} will illustrate, a sliced Weil decoration $\WD_E$ might satisfy $\WD_E|_U\equiv0$ without $E\otimes_\kk\CO_U=\CE|_U$, cf.\ Remark~\ref{rem:CanInc}.

\smallskip

(iii) Exactly as in \cite[Proposition 3.2]{TRS} one can show that the image of an $E$-\slice\ $\WD_E\colon E\to\hatDiv(E)\subseteq\hatDiv(X)$ is closed under $\wedge$ and has finite image (and not merely finite support).
\end{remark}

\subsection{The dual of a reflexive sheaf}
\label{subsec:DualRS}

As an application of slices we compute the Weil decoration of the dual $\CE^\vee$ of a reflexive sheaf $\CE$ in terms of $\WD_\CE$.

\begin{proposition}
\label{prop:Dual}
The Weil decoration $\WD_{\CE^\vee}\!\colon(\CE^\vee)_\eta=(\CE_\eta)^\vee\to\hatDiv(X)$ is given by
\begin{equation}
\label{eq:Wdual}
\WD_{\CE^\vee}(\varphi)=\bigwedge_{v\in\CE_\eta}\Big(\div\big(\varphi(v)\big)-\WD_\CE(v)\Big). 
\end{equation}
\end{proposition}

\begin{proof}
Let $\Delta_v$ be shorthand for the divisor $\div\big(\varphi(v)\big)-\WD_\CE(v)$. To start with,
\[
\underline D:=\bigwedge_{v\in\CE_\eta}\Delta_v=\bigwedge_{v\in\CE_\eta}\{\div\big(\varphi(v)\big)-\WD_\CE(v)\}
\]
is a well-defined divisor for $\varphi\in(\CE_\eta)^\vee$: Indeed, fix a slice $E$ with basis $e_1,\ldots,e_r$. If $\PD\in\sliceU_E$ and $v=\sum_{i=1}^rf_ie_i$, $f_i\in K$, (W1) and orthogonality over $\sliceU_E$ imply
\begin{align*}
(\Delta_v)_\PD
&\;\geq\;\min\limits_{i=1,\ldots,r}\big\{\div\!\big(f_i\varphi(e_i)\big)_\PD\big\}-\WD_\CE\big(\sum_{i=1}^rf_ie_i\big)_\PD\\
&\;=\;\min\limits_{i=1,\ldots,r}\{\div(f_i)_\PD+\div(\varphi(e_i))_\PD\}-\min\limits_{i=1,\ldots,r}\{\div(f_i)_\PD\}\\
&\;\geq\;\hspace{6pt}\min\{\div\big(\varphi(e_i)\big)_\PD\}.
\end{align*}
If $\PD_k$ is one of the finitely many prime divisors not in $\sliceU_E$ we fix a $\PD_k$-orthogonal basis $e_{k,1},\ldots,e_{k,r}$, and conclude as before that
\[
(\Delta_v)_{\PD_k}\geq\min\limits_{i=1,\ldots,r}\{\div\!\big(\varphi(e_{k,i})\big)_{\PD_k}\}.
\]
In particular, $\underline D_\PD\geq0$ for all but finitely many prime divisors. Thus, $\underline D$ is finitely supported, for $\underline D\leq\Delta_v$ if $v\in\CE_\eta$. Turning to the proof of~\eqref{eq:Wdual}, we first observe
\[
D\leq\WD_{\CE^\vee}(\varphi)\quad\text{if and only if}\quad\CO_X(D)\cdot\varphi\subseteq\CO_X(\WD_{\CE^\vee}(\varphi))\cdot\varphi=\KS\cdot\varphi\cap\CE^\vee
\]
as follows from~\eqref{eq:OD}. Since $\CE=\bigcup_{v\in\CE_\eta}\CO_X(\WD_\CE(v))\cdot v$, evaluating the right hand side in $\CE$ is equivalent to the following statement: For all $v\in\CE_\eta$, $f\in\CO_X(D)$ and $g\in\CO_X(\WD_\CE(v))$, we have $f\cdot\varphi(g\cdot v)\in\CO_X$, that is, $0\leq\div(f)+\div(g)+\div(\varphi(v))$. This, in turn, is equivalent to 
\[
0\leq-D-\WD_\CE(v)+\div(\varphi(v))=-D+\Delta_v
\]
for all $v\in\CE_\eta$. Therefore,
\[
D\leq\WD_{\CE^\vee}(\varphi)\quad\text{if and only if}\quad D\leq\underline D;
\]
in particular, $\underline D=\WD_{\CE^\vee}(\varphi)$. 
\end{proof}

\subsection{Toric \slice s}
\label{subsec:TS}
Slices also naturally appear for Weil decorations of toric reflexive sheaves, cf.~\cite{TRS}. First, we briefly fix our notation for present and later use. Let $X=\toric(\Sigma)$ be the toric variety over $\kk$ which is specified by a fan $\Sigma$; $\TT$ denotes the torus of $X$. The {\em character lattice} is given by algebraic group morphisms
\[
M=\Hom_{\opn{ag}}(\TT,\kk^*). 
\]
It induces the $\kk$-algebra $\kk[M]$ for which $\Spec\kk[M]=\TT$. For technical reasons, we discard some degenerate cases and {\em always} assume that the set of one-dimensional cones or {\em rays} $\Sigma(1)$ generates $N_\R=N\otimes_\Z\R$, where $N=\Hom_\Z(M,\Z)$ is the dual of $M$ providing a pairing 
$M\times N\to\Z$, $(m,n)\mapsto \MN{m}{n}=n(m)$.

\begin{example}
\label{exam:PnToric}
For $X=\PP^n$ we have $M=\Z^n$, and the fan is generated by the rays
\[
a_0=-\sum_{i=1}^ne_i,\;a_1=e_1,\;\ldots,\;a_n=e_n 
\]
for the standard basis $e_1,\ldots,e_n$ of $\Z^n$.
\end{example}

The fundamental sequence of toric geometry reads as
\begin{equation}
\label{eq:FundSeq}
\xymatrix{
0\ar[r]&M\ar[r]^-{\iota}&\Div_\TT(X)=\bigoplus_{\rho\in\Sigma(1)}\Z D_\rho\ar[r]^-{[\cdot]}&\Cl(X)\ar[r]&0.
}
\end{equation}
Here, $\Div_\TT(X)$ denotes the group of $\TT$-invariant Weil divisors freely generated by $\{D_\rho:=\overline{\opn{orb}(\rho)}\}_{\rho\in\Sigma(1)}$, the closures in $X$ of the $\TT$-orbits $\opn{orb}(\rho)$ corresponding to $\rho$. An element $m$ in $M$ is mapped to $\iota(m):=\sum_{\rho\in\Sigma(1)}
\MN{m}{\rho}D_\rho$ which equals the divisor of the rational function $x^m$ defined by $m$, while $[\cdot]$ sends a toric divisor to its class. As usual, we shall identify a ray with its primitive generator. A general reference for toric varieties is~\cite{CoxBook}. 

\medskip

Now let $\CE$ be a toric sheaf on $X$, that is, a reflexive sheaf $\CE$ with a linearised $\TT$-action. In particular, $\CE$ is already determined by the $M$-graded modules of sections $\CE(U_\sigma)$ over the torus invariant open affines $U_\sigma=\toric(\sigma)$, $\sigma\in\Sigma$. Taking 
\begin{equation}
\label{eq:EDef}
E:=\CE(\TT)_0=\Gamma(\TT,\CE)_0\cong\kk^r
\end{equation}
to be the $\kk$-vector space of $M$-degree $0$, that is, the torus invariant sections of $\CE$ over $\TT$, we see that $\CE(U_\sigma)$ sits naturally inside $\kk[M]\otimes_\kk E$. Further, $E$ defines a slice for $\WD_\CE$ to which we refer as {\em toric}. Since by equivariance, any $\kk$-basis $e_1,\ldots,e_r$ of $E$ trivialises $\CE$ over the open torus $\TT\subseteq X$, we have $\TT\subseteq \sliceU_E$; in particular, the $E$-slice of $\WD_\CE$ is supported on $\Div_\TT(X)$. Since for $0\neq e\in E$, 
\begin{equation}
\label{eq:MultiGrad}
\CO_X(\WD_E(e))(U_\sigma)\cdot e=(\KS\cdot e)\cap\CE(U_\sigma)=\big(\kk[M]\cdot e\big)\cap\CE(U_\sigma),
\end{equation}
$\WD_E$ is actually the Weil decoration of the toric sheaf $\CE$ in the sense of~\cite{TRS}; as such, it determines $\CE$ and therefore $\WD_\CE$. To see this explicitly, let us write
\[
\WD_E(e)=\sum_{\rho\in\Sigma(1)}b_\rho(e)D_\rho
\]
for $e\neq0$. Then $x^m\otimes e$ is in $\CE(U_\rho)\subseteq\kk[M]\otimes_\kk E$ if and only if $\MN{m}{\rho}\geq-b_\rho(e)$. For $\rho\in\Sigma(1)$ and $\ell\in\Z$, we recover the descending {\em Klyachko-filtration} $E^\ell_\rho:=\{e\in E\mid b_\rho(e)\geq\ell\}$~\cite{klyachko} for which we take a {\em $\rho$-adapted basis} $e_1,\ldots,e_r$ of $E$, that is, a basis compatible with the flag $E^\bullet_\rho$. If $m_i\in M$ is such that $\MN{m_i}{\rho}=-b_\rho(e_i)$, then
\[
\hat e_i:=x^{m_i}\cdot e_i\in\CE(U_\rho)  
\]
defines an $U_\rho$-orthogonal basis $\hat e_1,\ldots,\hat e_r$ of $\WD_\CE$. As $D_\rho\in U_\rho$, 
\begin{align*}
\WD_\CE\big(\sum f_ie_i\big)_{D_\rho}=\WD_\CE\big(\sum f_ix^{-m_i}\hat e_i\big)_\rho&=\min\{\ord_{D_\rho}(f_i)-\MN{m_i}{\rho}\mid i=1,\ldots,r\}\\
&=\min\{\ord_{D_\rho}(f_i)+b_\rho(e_i)\mid i=1,\ldots,r\}
\end{align*}
by Proposition~\ref{prop:CompDe}. We thus arrived at

\begin{proposition}
\label{prop:ToricSlicing}
Let $\CE$ be a toric sheaf over the toric variety $X=\toric(\Sigma)$ with toric slice $E$. Then $\TT\subseteq\sliceU_E$, and for a $\rho$-adapted basis $e_1,\ldots,e_r$ of $E$, $\rho\in\Sigma(1)$,
\[
\WD_\CE\big(\sum f_ie^i\big)_{D_\rho}=\min\{\ord_{D_\rho}(f_i)+\WD_E(e_i)_{D_\rho}\mid i=1,\ldots,r\}
\]
for all $f_1,\ldots,f_r\in K$.
\end{proposition}

\section{Morphisms of Weil decorations}
\label{sec:MorphWeilDecos}

\subsection{The category of Weil decorations}
\label{subsec:CatWeilDecos}
Our notion of morphism is this.

\begin{definition}
Let $\WD\colon\CV\to\hatDiv(X)$ and $\WD'\colon\CV'\to\hatDiv(X)$ be two Weil decorations. A {\em morphism $\mu\colon\WD\to\WD'$ between two Weil decorations $\WD$ and $\WD'$} is a $K$-linear map $\CV\to\CV'$ still denoted $\mu$ such that for all $v\in\CV$,
\[
\WD(v)\leq\WD'(\mu(v)).
\] 
\end{definition}

A morphism $\mu\colon\WD\to\WD'$ induces an $\CO_X$-module morphism $\OD\to\CO_X(\WD')$. This boosts the assignment from Proposition~\ref{prop:Bijection} into a functor
\[
F\colon\mathbf{WeilDeco}_X\to\mathbf{RefShe}_X,\quad\WD\mapsto\OD
\]
from the category of Weil decorations $\mathbf{WeilDeco}_X$ into the category $\mathbf{RefShe}_X$ of reflexive sheaves on $X$. Conversely, let $\CE$ and $\CE'$ be two reflexive sheaves on $X$. Each sheaf map $\mu\colon\CE\to\CE'$ induces a $K$-linear map $\mu_\eta\colon\CE_\eta\to\CE'_\eta$ between the generic stalks. Torsion-freeness makes the vertical maps in the commutative diagram 
\[
\xymatrix@!0@C=7em@R=7ex
{
\CE(U)\ar^-{\mu_U}[r]\ar@{^(->}[d] & 
\CE'(U)\ar@{^(->}[d]\\
\CE_\eta\ar^-{\mu_\eta}[r] & \CE'_\eta
}
\]
injective. Hence, we can reconstruct the sheaf map $\mu$ from $\mu_\eta$ alone by restricting to $\CE(U)$ to get a morphism $\WD_\CE\to\WD_{\CE'}$ of Weil decorations. This defines the functor  
\[
G\colon\mathbf{RefShe}_X\to\mathbf{WeilDeco}_X,\quad\CE\mapsto\WD_\CE,\quad G(\mu\colon\CE\to\CE'):=\big[\mu_\eta\colon\CE_\eta\to\CE'_\eta\big].
\]
Clearly, $F\circ G$ and $G\circ F$ are isomorphic to the identity functors on $\mathbf{RefShe}_X$ and $\mathbf{WeilDeco}_X$, respectively, so that we can upgrade Proposition~\ref{prop:Bijection} to

\begin{theorem}
\label{thm:EquivalenceCat}
The categories of $\mathbf{WeilDeco}_X$ and $\mathbf{RefShe}_X$ are equivalent.
\end{theorem}

\subsection{Kernels}
\label{subsec:KerWeilDecos}
For two reflexive sheaves $\CE'\subseteq\CE$ we have $\CE'_\eta\subseteq\CE_\eta$ as $K$-vector spaces whence $\WD_{\CE'}(v)\leq\WD_\CE(v)$ for $v\in\CE'_\eta$. Inequality can indeed occur, e.g., $\CO_X(-D)\subseteq\CO_X$ for any effective divisor $D$.

\begin{proposition}
\label{prop:KerWeilDecos}
Let $\CE'\subseteq\CE$ be reflexive sheaves with Weil decorations $\WD'$ and $\WD$, respectively. Let $\CE'':=\CE/\CE'$ be the cokernel and $\PD\in X$ be a prime divisor. Then
\begin{equation}
\label{eq:InjStalk}
\CE''_\PD\to\CE''_\eta\text{ is injective}\iff\WD'(v)_\PD=\WD(v)_\PD\text{ for all }v\in\CE'_\eta.
\end{equation}
In particular, $\WD'=\WD|_{\CE'_\eta}$ if $\CE'$ is the kernel of a morphism $\CE\to\CE''$ with $\CE''$ torsion-free.
\end{proposition}

\begin{proof}
If $\CE'=\ker(\CE\to\CE'')$ with $\CE''$ torsion-free, then $\CE''_x\hookrightarrow\CE''_\eta$ is injective for {\em all} points $x\in X$. The second statement is thus a direct implication of the equivalence in~\eqref{eq:InjStalk}. To prove the latter we let $0\to\CE'\to\CE\stackrel{\mu}{\to}\CE''\to0$ be a short exact-sequence of $\CO_X$-modules with $\CE'$ and $\CE$ reflexive. An arbitrary prime divisor $\PD$ of $X$ yields the commutative diagram
\[
\xymatrix@!0@C=5em@R=7ex{
0\ar[r]&\CE'_\PD\ar[r]\ar@{^(->}[d]&\CE_\PD\ar^-{\mu_\PD}[r]\ar@{^(->}[d]&\CE''_\PD\ar[d]^-{\nu_\PD}\ar[r]&0\\
0\ar[r]&\CE'_\eta\ar[r]&\CE_\eta\ar^-{\mu_\eta}[r]&\CE''_\eta\ar[r]&0
}
\]
with exact rows and two injective vertical maps. A standard diagram chase reveals \[
\nu_\PD\;\text{is injective}\iff\CE'_\eta\cap\CE_\PD=\CE'_\PD,
\]
where intersection takes place in $\CE_\eta$. Assuming that $\nu_\PD\colon\CE''_\PD\to\CE''_\eta$ is injective,
\[
K\cdot v\cap\CE'_\PD=K\cdot v\cap(\CE'_\eta\cap\CE_P)=K\cdot v\cap\CE_P  
\]
for all $v\in\CE'_\eta$ whence $\WD'(v)_\PD=\WD(v)_\PD$. Conversely, $\CE'_\eta\cap\CE_\PD\subsetneq\CE'_\PD$ implies that there exists $v\in\CE'_\eta\cap\CE_\PD$ with $\WD(v)_\PD\geq0$, but $\WD'(v)_\PD<0$.
\end{proof}

\begin{remark}
Injectivity of all maps $\CE''_\PD\to\CE''_\eta$ is {\em not} enough to guarantee torsionfreeness of $\CE_\eta$. For instance, take an integral domain $A$ and a non-prime ideal $I$ not contained in any height one prime. Then $A/I$ has torsion, but $(A/I)_\PD=0$ for any height one prime $\PD$. 
\end{remark}

\subsection{Quotients}
\label{subsec:QuotWeilDecos}
Next we turn to quotients.

\begin{proposition}
\label{prop:QuotWeilDecos}
Let $\mu\colon\CE\to\CE'$ be a morphism between reflexive $\CO_X$-modules with Weil decorations $\WD$ and $\WD'$, and let $\PD\in X$ be a prime divisor. Then
\begin{center}
$\mu_\PD\colon\CE_\PD\surj\CE'_\PD$ is surjective $\iff$ 
\begin{tabular}[t]{l}
for all $e'\in\CE'_\eta$, the $\PD$-coefficient of $\WD'(e')$ is\\[2pt]
$\WD'(e')_\PD=\max\{\WD(e)_\PD\mid e\in\mu_\eta^{-1}(e')\subseteq\CE_\eta\}$.
\end{tabular} 
\end{center}
In particular, if $\mu\colon\CE\to\CE'$ is surjective, then $\WD'(e')=\bigvee\limits_{e\in\mu_\eta^{-1}(e')}\WD\big(e\big)$.
\end{proposition}

\begin{proof}
Let $e'\in\CE'_\eta$ and assume first that for a prime divisor $\PD$ in $X$, $\mu_\PD$ is surjective. For any $e\in\mu^{-1}(e')$ we have $\WD(e)_\PD\leq\WD'(e')_\PD$; equality is attained for at least one $e$. Indeed, if $t\in\idm_{X,\PD}\setminus\idm_{X,\PD}^2$ is a local parameter for the DVR $\CO_{X,\PD}$ and $k:=\WD'(e')_\PD\in\Z$, then $\WD'(t^{-k}e')_\PD=0$, that is, $t^{-k}e'\in\CE'_\PD$. By surjectivity there is a $\tilde e\in\CE_\PD$ which maps to $t^{-k}e'$ whence
\[
0\;\leq\;\WD(\tilde e)_\PD\;\leq\;\WD'(t^{-k}e')_\PD\;=\;0
\]
and so $\WD(\tilde e)_\PD=0$. As a result, $t^k\tilde e\in\mu_\eta^{-1}(e')$ and
$\WD(t^k\tilde e)_\PD=k=\WD'(e')_\PD$. Conversely, let $e'\in\CE'_\PD\subseteq\CE'_\eta$. Then there exists a $e\in\CE_\eta$ with $\mu_\eta(e)=e'$ and $\WD_\CE(e)_\PD=\WD_\CE'(e')_\PD\geq0$. Hence $e\in\CE_\PD$ so that $\mu_P$ is surjective.
\end{proof}

\begin{corollary}
\label{coro:DirSum}
The Weil decoration of the direct sum is given by
\[
\WD_{\CE\oplus\CE'}(e\oplus e')=\WD_\CE(e)\wedge\WD_{\CE'}(e').
\]
\end{corollary}

\subsection{The Euler sequence}
\label{subsec:EulSeq} 
We illustrate our methods by considering the Euler sequence of a smooth toric variety $X=\toric(\Sigma)$, namely
\begin{equation} 
\label{eq:9DiagDR}
\begin{tikzcd}[column sep=17pt]
0\ar[r]&\Omega_X\ar[r,"\iota_X"]&
\bigoplus\limits_{\rho\in\Sigma(1)}\CO_{X}(-D_\rho)\ar[r]&
\Cl(X)\otimes\CO_{X}\ar[r]&0.
\end{tikzcd}
\end{equation}
Now $\Omega_X|_\TT$ becomes the sheaf $M\otimes_\Z\kk[M]$ by sending $\dd x^m$ to $m\otimes x^m$ whence 
\[
(\Omega_X)_\eta=M_K=M\otimes_\Z K,  
\]
cf.~\cite[8.1.2]{CoxBook}. In fact, taking stalks at the generic point yields the fundamental sequence~\eqref{eq:FundSeq} tensored by $\otimes_\Z K$, that is, 
\begin{equation}
\label{eq:FundSeqK} 
\begin{tikzcd}
0\ar[r]&M_K\ar[r,"\iota_K"]&K^{\Sigma(1)}\ar[r,"{[}\cdot{]_K}"]&\Cl(X)\otimes K\ar[r]&0.
\end{tikzcd}
\end{equation}
By Proposition~\ref{prop:KerWeilDecos}\,, $\WD_{\Omega_X}$ is the restriction of the Weil decoration of $\bigoplus_\rho\CO_X(-D_\rho)$ to the generic stalk $(\Omega_X)_\eta$ whence
\begin{equation}
\label{eq:WDOmega}
\WD_{\Omega_X}(m)=
\bigwedge\limits_{\rho\in\Sigma(1)}\!\!\Big(\!\div(\MN{m}{\rho})-D_\rho\Big)
\end{equation}
for all $m\in M_K$. Since $\Gamma(\TT,\Omega_X)=M_\kk\otimes_\kk\kk[M]$, the toric slice is $M_\kk=M\otimes_\Z\kk$. By Proposition~\ref{prop:ToricSlicing}, \begin{equation}
\label{eq:TSOmegaX}
(\WD_{\Omega_X})_{M_\kk}(m)=-\sum_{\MN{m}{\rho}\neq 0}D_\rho
\end{equation}
for $m\in M_\kk$, in accordance with~\cite[4.9]{TRS}.

\medskip

Next, we turn to the Weil decoration of the tangent sheaf $\CT_X$. First, some further toric terminology is in order. Let $\Sigma(d)$ denote the subset of $d$-dimensional cones in $\Sigma$, and let $\sigma(1)\subseteq\Sigma(1)$ be the set of rays contained in $\sigma$. By smoothness, they define a basis of $N$ if $\sigma\in\Sigma(n)$ so that
\[
\piN_\sigma\colon\Z^{\sigma(1)}\hookrightarrow\Z^{\Sigma(1)}\stackrel{\piN}{\twoheadrightarrow}N
\]
extends to an isomorphism $K^{\sigma(1)}\cong N_K$. In particular, we can assign to any $a\in N_K$ a uniquely determined element $\taSig:=\piN_\sigma^{-1}(a)\in K^{\sigma(1)}$. If, for any $\rho\in\sigma(1)$, we let $\rhoSig\in\sigma^\vee(1)$ be the element of the dual basis of $\sigma(1)$ with $\MN{\rhoSig}{\rho}=1$, then the $\rho$-coordinate of $\taSig$ in $K^{\sigma(1)}$ is given by $\taSig_\rho=\MN{\rhoSig}{a}$. Proposition~\ref{prop:Dual} implies via an explicit computation the following formula.

\begin{proposition}
\label{prop:TXWD}
Let $n=\dim X$. Then for $a\in N_K$ we have
\begin{equation}
\label{eq:TangBdlWD}
\WD_{\CT_X}(a)=\bigwedge_{\sigma\in\Sigma(n),\,\rho\in\sigma(1)}\!\! 
\Big(\!\div \MN{\rhoSig}{a} + D_\rho + \hspace{-1em}
\sum_{\rho'\in\Sigma(1)\setminus\sigma(1)} \hspace{-1.5em}D_{\rho'}\Big).
\end{equation}
Equivalently, for any fixed $\sigma\in\Sigma(n)$ and prime divisor $\PD\in U_\sigma$ we have
\begin{equation}
\label{eq:TangBdlWDForPD}
\WD_{\CT_X}(a)_\PD=\min_{\rho\in\sigma(1)}\big(\ord_\PD\MN{\rhoSig}{a} + \delta_{D_\rho,\PD}\big).
\end{equation}
\end{proposition}

\begin{proof}
Let $\PD\in X$ be a prime divisor. We proceed in several steps.

\smallskip

{\bf Step 1:}
Applied to $\bigoplus_{\rho\in\Sigma(1)}\CO_X(D_\rho)\twoheadrightarrow\CT_X$, Proposition~\ref{prop:QuotWeilDecos} yields
\begin{align*}
\WD_{\CT_X}(a)_\PD&=\max\limits_{\ta\mapsto a}\WD_{\bigoplus\CO(D_\rho)}(\ta)_\PD
\;=\;\max\limits_{\ta\mapsto a}\min\limits_{\rho\in\Sigma(1)}\WD_{\CO(D_\rho)}(\ta_\rho)_\PD\\
&=\max\limits_{\ta\mapsto a}\min\limits_{\rho\in\Sigma(1)}\big(\ord_\PD\ta_\rho + \delta_{D_\rho,\PD}\big)\;=:\;\fbox{$\max\limits_{\ta\mapsto a}\Phi(\ta)_\PD$}\,.
\end{align*}
If $\PD\in U_\sigma$ for $\sigma\in\Sigma(n)$, we consider the special preimage $\taSig\in K^{\sigma(1)}$ of $a$. This results in $\WD_{\CT_X}(a)_\PD\geq \Phi\big(\taSig\big)_\PD$ and
\begin{align}
\Phi\big(\taSig\big)_\PD&=\min\limits_{\rho\in\Sigma(1)}\big(\ord_\PD\taSig_\rho+\delta_{D_\rho,\PD}\big)\;=\min\limits_{\rho\in\sigma(1)}\big(\ord_\PD\taSig_\rho+\delta_{D_\rho,\PD}\big)\nonumber\\ 
&=\min\limits_{\rho\in\sigma(1)}\big(\ord_\PD\MN{\rhoSig}{a}+\delta_{D_\rho,\PD}\big)\label{eq:PhiPD}.
\end{align}
The proof of the second formula~\eqref{eq:TangBdlWDForPD} boils down to the claim $\Phi\big(\taSig\big)_\PD=\WD_{\CT_X}(a)_\PD$.

\smallskip

{\bf Step 2:}
Towards this end we use Proposition~\ref{prop:Dual} which implies
\[
\WD_{\CT_X}(a)_\PD=\min\limits_{m\in M_K}\!\big(\ord_\PD\MN{m}{a}-\WD_{\Omega_X}\!(m)_\PD\,\big).
\]
In particular, we conclude
\[
\Phi\big(\taSig\big)_\PD\;\leq\;\max\limits_{\ta\mapsto a}\Phi(\ta)_\PD\;=\;\WD_{\CT_X}(a)_\PD\;\leq\;\ord_\PD\!\MN{m}{a}-\WD_\Omega(m)_\PD
\]
for all $m\in M_K$.

\smallskip

{\bf Step 3:}
We finish by exhibiting an $m\in M_K$ with $\Phi\big(\taSig\big)_\PD=\ord_\PD\!\MN{m}{a}-\WD_\Omega(m)_\PD$. Let $\rhoSt\in\sigma(1)$ be a ray realising the minimum in~\eqref{eq:PhiPD} for $\Phi\big(\taSig\big)_\PD$, i.e., 
\[
\Phi\big(\taSig\big)_\PD=\ord_\PD\MN{\rhoStSig}{a}+\delta_{D_{\rhoSt},\PD}.
\]
We put $m:=\rhoStSig\in\sigma^\vee(1)\subseteq M$ and obtain
\begin{align}
\Phi\big(\taSig\big)_\PD-\ord_\PD\!\MN{m}{a}+\WD_\Omega(m)_\PD&=
\Phi\big(\taSig\big)_\PD-\ord_\PD\!\MN{\rhoStSig}{a}+\WD_\Omega(\rhoStSig)_\PD\nonumber\\ 
&=\delta_{D_\rhoSt,\PD}+\WD_\Omega(\rhoStSig)_\PD\nonumber\\ 
&=\delta_{D_\rhoSt,\PD}+\min_{\rho\in\Sigma(1)}\!\big(\ord_\PD\!\MN{\rhoStSig}{\rho}-\delta_{D_\rho,\PD}\big).\label{eq:MinPhi}
\end{align}
Now $\PD\in\TT$ entails $\delta_{D_\rhoSt,\PD}=\delta_{D_\rho,\PD}=0$, and the pairings $\MN{\rhoStSig}{\rho}$ are constant but not simultanously zero, e.g.\ for $\rho=\rhoStSig$. Hence the minimum in~\eqref{eq:MinPhi} is finite and equals zero. On the other hand, if $\PD=D_\mu$ for a $\mu\in\Sigma(1)$, then $\PD\in U_\sigma$ implies $\mu\in\sigma(1)$. Furthermore, 
\[
\ord_\PD\!\MN{\rhoStSig}{\rho}-\delta_{\rho,\mu}=
\begin{cases}
\ord_\PD\!\MN{\rhoStSig}{\rho}-0\geq0,&\rho\neq\mu,\\
\infty-1,&\rho=\mu\text{ and }\rho\neq\rhoSt,\\
-\delta_{\rho,\mu}=-1,&\rhoSt=\rho=\mu.
\end{cases}
\]
Hence, the minimum of this equals $-\delta_{D_\rhoSt,\PD}$ which cancels the first term in~\eqref{eq:MinPhi}.

\smallskip

{\bf Step 4:}
It remains to check~\eqref{eq:TangBdlWD}. For a fixed $\sigma\in\Sigma(n)$, 
we let
\[
\Psi(a,\sigma)_\PD:=
\min_{\rho\in\sigma(1)}\big(\div \MN{\rhoSig}{a} + D_\rho + \hspace{-0em}
\sum_{\rho'\notin\sigma(1)} \hspace{-0em}D_{\rho'}\big)_\PD.
\] 
If $\PD\in U_\sigma$, then $\Psi(a,\sigma)_\PD=\WD_{\CT_X}(a)_\PD$ 
by the formula~\eqref{eq:TangBdlWDForPD} established in Step 3.

\smallskip

On the other hand, $\PD\notin U_\sigma$ entails 
$\Psi(a,\sigma)_\PD=\min_{\rho\in\sigma(1)}\big(\ord_{\PD}\MN{\rhoSig}{a}+1\big)$. 
Now for any cone $\sigma'$ with $\PD\in U_{\sigma'}$ we can express each element $\rhoSig\in\sigma^\vee(1)$ in terms of the $\Z$-basis $(\sigma')^\vee(1)\subseteq M$. The usual valuation properties then imply
\[
\ord_{\PD}\MN{\rhoSig}{a}+1\geq \Psi(a,\sigma')_\PD=\WD_{\CT_X}(a)_\PD.
\]
In particular, $\Psi(a,\sigma)_\PD\geq \Psi(a,\sigma')_\PD$ does not contribute to the meet in~\eqref{eq:TangBdlWD}.
\end{proof}

\begin{remark}
It is straightforward to check that specialising~\eqref{eq:TangBdlWD} to the toric slice $N_\kk\subseteq N_K$ recovers the formula 
\[
(\WD_{\CT_X})_{N_k}(v)=\sum_{\rho\in\spann(v)}D_\rho
\]
for $v\in N_\kk$ from~\cite[4.6]{TRS} for the toric Weil decoration of $\CT_X$.
\end{remark}

\section{The Horrocks-Mumford bundle}
\label{sec:HMBundle}
In this section we will determine the Weil decoration of the Horrocks-Mumford bundle, subsequently referred to as {\em $\HM$-bundle}. In a way, this is the ``most toric'' non-toric sheaf, and this will be reflected in its Weil decoration.

\subsection{The monad construction of the \texorpdfstring{$\HM$}{HM}-bundle}
\label{subsec:Hulek}
Among the various constructions of the $\HM$-bundle, cf.\ \cite{HulekHM}, we presently review the original one via monads from~\cite{hm}. Since we solely work with $X=\PP^4$, we simply write $\CO$ for $\CO_{\PP^4}$ etc. Further, we define for any $\kk$-vector space $W$ the sheaf 
\[
W(\ell):=W\otimes_\kk\CO(\ell). 
\]
Consider $V:=\kk^5$ with its standard basis $\{e_0,\ldots,e_4\}=\{e_\nu\mid
\nu\in\Z/5\Z\}$ and the associated projective space $\PP^4=\PP(V)$. The dual basis $\{z_0,\ldots,z_4\}=\{z_\nu\mid\nu\in\Z/5\Z\}$ of $V^*$ induces homogeneous coordinates on $\PP^4$ and defines the hyperplanes \fbox{$H_\nu=\{z_\nu=0\}$}. Subsequently, we use the natural identifications
\[
\textstyle V=\spann_\kk\{\frac{\partial}{\partial z_0},\ldots,\frac{\partial}{\partial z_4}\}\quad\text{and}\quad V^*=\spann_\kk\{\dd z_0,\ldots,\dd z_4\}.
\]
A prominent role is played by the section
\[
s=\sum_{\nu}\frac{\partial}{\partial z_\nu}\otimes z_\nu\in\kG(\PP^4,V(1)).
\]
First, $s$ induces the morphism $\CO\to V(1)$ in the dual Euler sequence
\begin{equation}
\label{eq:DualEuler}
\begin{tikzcd}[column sep=huge]
0\ar[r]&\CO\ar[r,"s"]&V(1)\ar[r,"\iota^*"]&\CT\ar[r]&0.
\end{tikzcd}
\end{equation}
Second, $s$ appears in the Koszul complex
\[
0\to\CO\stackrel{s}{\longrightarrow}V(1)\stackrel{\wedge s}{\longrightarrow}\fbox{$(\Lambda^2V)(2)\stackrel{\wedge s}{\longrightarrow}(\Lambda^3V)(3)$}\stackrel{\wedge s}{\longrightarrow}(\Lambda^4V)(4)\to(\Lambda^5V)(5)\to0.
\]
The map in the framed box above factorises via
\begin{equation}
\label{eq:MapFB}
\begin{tikzcd}[column sep=large]
\hspace{-10pt}(\Lambda^2V)(2)\ar[r,twoheadrightarrow,"p_0:=\Lambda^2\rho"]&\Lambda^2\CT\ar[r,hookrightarrow,"q_0:=p^*_0\circ\Phi"]&(\Lambda^2V^*)(-2)\otimes_\CO\CO(\partial\PP^4)\cong\big(\Lambda^3V\big)(-2)(5),
\end{tikzcd}
\end{equation}
where $\partial\PP^4=\sum_{\nu=0}^4H_\nu$, and $\Phi$ is the isomorphism identifying $\Lambda^2\CT\otimes\CO(-\partial\PP^4)$ with $(\Lambda^2\CT)^*=\Omega^2$. Furthermore, $\Lambda^5V=\kk$ as $V$ comes with a distinguished basis. Now define the linear maps
\[
f^\pm\colon V\to\Lambda^2V,\quad f^+(\tfrac{\partial}{\partial z_\nu})=\tfrac{\partial}{\partial z_{\nu+2}}\wedge\tfrac{\partial}{\partial{z_{\nu-2}}}\text{ and }f^-(\tfrac{\partial}{\partial z_\nu})=\tfrac{\partial}{\partial{z_{\nu+1}}}\wedge\tfrac{\partial}{\partial{z_{\nu-1}}},\;\nu\in\Z/5\Z.
\]
Denoting \fbox{$f_{\pm}=f^{\pm*}$} the maps dual to $f^{\pm}$,~\eqref{eq:MapFB} fits into the diagram
\begin{equation}
\label{eq:Hulek}
\begin{tikzcd}[column sep=14pt]
&(\Lambda^2V)(2)\ar[r,"p_0"]&\Lambda^2\CT\ar[r,"q_0"]&(\Lambda^3V)(-2)\otimes_\CO\CO(\partial\PP^4)\ar[rd,"f_-"]&\\
V(2)\ar[ru,"f^+"]\ar[rd,"f^-"']&&&&\text{\hspace{-35pt}}V^*(-2)\otimes_\CO\CO(\partial\PP^4).\\
&(\Lambda^2V)(2)\ar[r,"p_0"]&\Lambda^2\CT\ar[r,"q_0"]&(\Lambda^3V)(-2)\otimes_\CO\CO(\partial\PP^4)\ar[ru,"f_+"']
\end{tikzcd}
\end{equation}
Its commutativity will follow from Diagram~\eqref{eq:GenStalkDia2} together with Equations~\eqref{eq:UpperPath} and~\eqref{eq:LowerPath} below. Ultimately, the morphisms
\[
p=(p_0\circ f^+)\oplus(p_0\circ f^-)\quad\text{and}\quad q=(f_-\circ q_0)\oplus(-f_+\circ q_0)
\]
(note the sign before $f_+$!) lead to the monad
\begin{equation}
\label{eq:Monad}
V(2)\;\stackrel{p}{\hookrightarrow}\;\Lambda^2\CT\oplus\Lambda^2\CT\;\stackrel{q}{\surj}\;V^*(-2)\otimes_\CO\CO(\partial\PP^4)
\end{equation}
whose cohomology defines the Horrocks-Mumford bundle $\HM$.

\begin{theorem}
\label{thm:WDofHM}
We have a canonical isomorphism $K^2\stackrel{\sim}{\longrightarrow}\HM_\eta$, and the Weil decoration of $\HM$ is induced by the family of semi-norms $\varphi_{h_\nu,H_\nu}$ given by
\[
h_\nu=\tfrac{z_{\nu+1}z_{\nu-1}}{z_{\nu+2}z_{\nu-2}}\in\kappa(H_\nu)^*,\quad\nu\in\Z/5\Z,  
\]
cf.\ Proposition~\ref{prop:SemiNorm}. Explicitely, we have
\[
\renewcommand{\arraystretch}{1.3}
\WD_\HM(f,g)_\PD=
\begin{cases}
\min\{\ord_\PD(f),\ord_\PD(g)\}+1,&\PD=H_\nu\,\text{and}\,
\tfrac fg(H_\nu)=h_\nu\\
\min\{\ord_\PD(f),\ord_\PD(g)\},&\text{else}.
\end{cases}
\]
\end{theorem}

Proving Theorem~\ref{thm:WDofHM} will occupy us for the remainder of this section. The monad description naturally lends itself to a simple divide and conquer strategy: We split the computation into linear algebra (the $K$-vector spaces provided by the generic stalks) and an optimisation problem (determination of a maximum).

\subsection{Linear algebra}
\label{subsec:LinAlg}
To understand the generic stalk of the invertible sheaves $\CO(\ell)$, $\ell\in\Z$, let $\uz$ be shorthand for $(z_0,\ldots,z_4)$ and consider the {\em rational Cox ring}
\[
\QCox:=\{f(\uz)/g(\uz)\mid f,\,g\in\kk[\uz]\text{ are homogeneous, }g\neq0\}\subseteq\kk(\uz),
\]
which is a $\Z$-graded vector space over the field
\[
K=K(\PP^n)=\kk[\uz]_{((0))}=\kk(z_i/z_j\mid i,j=0,\ldots,4)=:\QCox_0.
\]
Every homogeneous component $\QCox_\ell$, $\ell\in\Z$, is a one-dimensional $K$-vector space and comprises the monomials $z_\nu^\ell$, $\nu\in\Z/5\Z$. Choosing, say $z=z_0$, gives the explicit representation
\[
\QCox=\bigoplus_{\ell\in\Z}\,\QCox_\ell=\bigoplus_{\ell\in\Z}\, K(\PP^n)\cdot z^\ell=K(\PP^n)[z,z^{-1}].
\]
Under this identification, the generic stalk $\CO(\ell)_\eta$ becomes $\QCox_\ell$, and
\[
\begin{tikzcd}
\CO(\ell\cdot H_0)\ar[d,"\cong"]\ar[r,hook]&K\ar[d,"\cdot z_0^\ell","\cong"']\\
\CO(\ell)\ar[r,hook]&\QCox_\ell 
\end{tikzcd}
\]
implies for instance
\begin{equation}
\label{eq:WDOl} 
\WD_{\CO(\ell)}(z_0^\ell)=\ell\cdot H_0. 
\end{equation}
In contrast, the generic stalk of the embedded invertible sheaf $\CO(\partial\PP^n)\hookrightarrow K$ occuring in~\eqref{eq:Hulek} is simply $K$. Thus, passing to generic stalks in Diagram~\eqref{eq:Hulek} renders the contribution of $\CO(\partial\PP^4)$ invisible and yields
\begin{equation}
\label{eq:GenStalkDia2}
\begin{tikzcd}[column sep=15pt]
&\Lambda^2V_K(2)\ar[r,"p_0"]&\Lambda^2N_K\ar[r,"\Phi","\sim"']&\Lambda^2M_K\ar[r,"p_0^*"]&\Lambda^2V^*_K(-2)\ar[rd,"f_-"]\\
V_K(2)\ar[ru,"f^+"]\ar[rd,"f^-"']&&&&&V^*_K(-2).\\
&\Lambda^2V_K(2)\ar[r,"p_0"]&\Lambda^2N_K\ar[r,"\Phi","\sim"']&\Lambda^2M_K\ar[r,"p_0^*"]&\Lambda^2V^*_K(-2)\ar[ru,"f_+"']
\end{tikzcd}
\end{equation}
Here, $M_K=M\otimes_\Z K$ and $N_K=N\otimes_\Z K$ are the generic stalks of the sheaf of differential forms $\Omega$ and the tangent sheaf $\CT$, cf.\ \ssect{subsec:EulSeq}, and we defined
\[
V_K(\ell):=V(\ell)_\eta=V\otimes_\kk\CO(\ell)_\eta=V\otimes_\kk\QCox_\ell.
\]
Finally, $\Phi\colon\Lambda^2N\stackrel{\sim}{\longrightarrow}\Lambda^2M$ is the natural isomorphism coming from $\Lambda^4N=\Z$ after the choice of an orientation. To fix one we start with the natural $K$-basis
\[
\{H_0,\ldots,H_4\}\subseteq\Div_\TT(\PP^4) 
\]
giving $K^5=\Div_\TT(\PP^4)\otimes_\Z K$. The map $\iota^*_K\colon K^5\to N_K$ obtained by dualising the sequence~\eqref{eq:FundSeqK} sends $H_\nu$ to the rays $a_\nu\in N$, $\nu\in\Z/5\Z$, cf.\ Example~\eqref{exam:PnToric}. This yields the ordered $K$-bases $\{a_1,a_2,a_3,a_4\}\subseteq N$ for $N_K$ and
\begin{equation}
\label{eq:Lambda2NKBasis}
\{[12],\; [13],\; [14],\; [23],\; [24],\; [34]\},\quad\text{where }[ij]:=a_i\wedge a_j,
\end{equation}
for $\Lambda^2N_K$. Finally, we let \fbox{$\partial_\nu:=z_\nu\tfrac{\partial}{\partial z_\nu}$} and take $z_\nu^{\ell-1}\partial_\nu=z_\nu^\ell\tfrac{\partial}{\partial{z_\nu}}$, $\nu\in\Z/5\Z$, as a $K$-basis for $V_K(\ell)$ whence
\begin{equation}
\label{eq:EulerSurj}
(\iota^*)_\eta=\iota^*_K\colon V_K(1)\twoheadrightarrow\CT_\eta=N_K,\quad\partial_\nu\mapsto a_\nu,\;\nu\in\Z/5\Z,
\end{equation}
cf.~\eqref{eq:DualEuler}. Then $p_0\circ f^+\colon V_K(2)\to\Lambda^2N_K$ is represented by the $(6\times 5)$ matrix
\begin{equation}
\label{eq:TableA}
\renewcommand{\arraystretch}{1.7}
A:=
\begin{array}{|ccccc|c}
z_0\partialP_0&z_1\partialP_1&z_2\partialP_2&z_3\partialP_3&
z_4\partialP_4&\\[0.3ex]
\hline
&&&\frac{z_3^2}{z_0z_1}&\frac{z_4^2}{z_1z_2}&[12]\\
&&&\frac{z_3^2}{z_0z_1}&&[13]\\
&&\frac{z_2^2}{z_0z_4}&\frac{z_3^2}{z_0z_1}&&[14]\\
\frac{z_0^2}{z_2z_3}&&&&&[23]\\
&&\frac{z_2^2}{z_0z_4}&&&[24]\\
&\frac{z_1^2}{z_3z_4}&\frac{z_2^2}{z_0z_4}&&&[34]\\
\hline
\end{array} 
\end{equation}
For instance, the third column $\frac{z_2^2}{z_0z_4}\cdot([14]+[24]+[34])$ is obtained from 
\[
z_2\partialP_2=z_2^2\frac{\partial}{\partial z_2}\mapsto z_2^2\cdot(\frac{\partial}{\partial z_4}\wedge \frac{\partial}{\partial z_0})=\frac{z_2^2}{z_0z_4}\cdot (\partialP_4\wedge\partialP_0)\mapsto\frac{z_2^2}{z_0z_4}\cdot \big(a_4\wedge (-a_1-a_2-a_3)\big).
\]
We rewrite $A$ as the product (again with zeroes omitted)
\begin{equation}
\label{eq:A0}
\fbox{$A=A_0\cdot D_A$}:=
\begin{pmatrix}
&&&1&1\\
&&&1&\\
&&1&1&\\
1&&&&\\
&&1&&\\
&1&1&&
\end{pmatrix}
\cdot\diag(\tfrac{z_0^2}{z_2z_3},\tfrac{z_1^2}{z_3z_4},\tfrac{z_2^2}{z_0z_4},\tfrac{z_3^2}{z_0z_1},\tfrac{z_4^2}{z_1z_2}).
\end{equation}
Similarly, we obtain for $p_0\circ f^-\colon V_K(2)\to\Lambda^2N_K$ the $(6\times 5)$ matrix
\begin{equation}
\label{eq:B0}
\arraycolsep=4pt
\fbox{$B=B_0\cdot D_B$}:=\left(\begin{array}{rrrrr}
&1&&&\\
&&-1&&-1\\
1&&&&\\
&-1&&&-1\\
&-1&&-1&\\
&&&&1
\end{array}\right)
\cdot\diag(\tfrac{z_0^2}{z_1z_4},\tfrac{z_1^2}{z_0z_2},\tfrac{z_2^2}{z_1z_3},\tfrac{z_3^2}{z_2z_4},\tfrac{z_4^2}{z_0z_3}).
\end{equation}
Finally, with respect to the $\Z$-basis of $\Lambda^2N$ from ~\eqref{eq:Lambda2NKBasis} and its induced dual basis in $\Lambda^2M$, the isomorphism $\Phi\colon\Lambda^2N\stackrel{\sim}{\to}\Lambda^2 M$ is given by the anti-diagonal matrix
\[
\Phi=\antidiag(1,-1,1,1,-1,1).
\]
Since $B^\top_0\Phi A_0=\Id_5$ is the identity matrix, the upper path of diagram~\eqref{eq:GenStalkDia2} leads to
\begin{equation}
\label{eq:UpperPath}
B^\top\Phi A=D_B(B^\top_0\Phi A_0)D_A=D_BD_A=\tfrac{1}{z_0\ldots z_4}\cdot\diag(z_0^5,z_1^5,z_2^5,z_3^5,z_4^5):=D.
\end{equation}
Similarly, the lower path of~\eqref{eq:GenStalkDia2} yields
\begin{equation}
\label{eq:LowerPath}
A^\top\Phi B=(B^\top \Phi^\top A)^\top=\tfrac{1}{z_0\ldots z_4}\cdot\diag(z_0^5,z_1^5,z_2^5,z_3^5,z_4^5)=D.
\end{equation}
As a result, Diagram~\eqref{eq:GenStalkDia2} is commutative, and so is therefore Diagram~\eqref{eq:Hulek}. The explicit matrix representations $A$ and $B$ for the maps $p_0\circ f^+\colon V_K(2)\to\Lambda^2 N_K$ and $p_0\circ f^-\colon V_K(2)\to\Lambda^2 N_K$ also yield the matrices \fbox{$F_-:=B^\top\Phi$} and \fbox{$F_+:=A^\top\Phi$} of $f_-\circ p^*_0\circ\Phi\colon\Lambda^2 N_K\to V^*_K(-2)$ and $f_+\circ p^*_0\circ\Phi\colon\Lambda^2 N_K\to V^*_K(-2)$, respectively. Hence, the generic stalk $\HM_\eta$ is the cohomology of
\begin{equation} 
\label{eq:MonadGS}
\begin{tikzcd}[column sep=huge]
V_K(2)\ar[r,"{(A,B)^\top}"]&\Lambda^2N_K\oplus\Lambda^2N_K\ar[r,"{(F_-,-F_+)}"]&V^*_K(-2)
\end{tikzcd}
\end{equation}
(still note the sign before $F_+$). The one-dimensional kernels of the matrices 
\[
\arraycolsep=4pt
A_0^\top\Phi=
\newcommand{\kre}{{\color{red}1}}
\newcommand{\krn}{{\color{red}0}}
\newcommand{\krm}{{\color{red}-1}}
\left(\begin{array}{rrrrrr}
0&\krn&1&\krn&\krn&0\\
1&\krn&0&\krn&\krn&0\\
1&\krm&0&\kre&\krn&0\\
0&\krn&0&\kre&\krm&1\\
0&\krn&0&\krn&\krn&1
\end{array}\right)
\quad\text{and}\quad
B_0^\top\Phi=
\renewcommand{\kre}{{\color{darkgreen}1}}
\renewcommand{\krn}{{\color{darkgreen}0}}
\renewcommand{\krm}{{\color{darkgreen}-1}}
\left(\begin{array}{rrrrrr}
\krn&0&\krn&1&0&\krn\\
\krn&1&\krm&0&0&\kre\\
\krn&0&\krn&0&1&\krn\\
\krn&1&\krn&0&0&\krn\\
\kre&0&\krm&0&1&\krn
\end{array}\right)
\]
which appear in $F_+=A^\top\Phi=D_A\cdot A_0^\top\Phi$ and $F_-=B^\top\Phi=D_B\cdot B_0^\top\Phi$, respectively, are generated by
\begin{equation} 
\label{eq:AlphaPM}
\fbox{$
\alpha^+:= {\color{red}[13]+[23]+[24]}\in \Lambda^2 N$}\quad\text{and}\quad\fbox{$\alpha^-:= {\color{darkgreen}[12]+[14]+[34]}\in\Lambda^2N$}\,.
\end{equation}
These vectors define the two-dimensional subspace  
\[
U_\alpha:=\spann_K\{(0,\newMinus\alpha^+),(\alpha^-,0)\}\subseteq\Lambda^2N_K\oplus\Lambda^2N_K.
\]
which is transversal to the $5$-dimensional space 
$V_K(2)\stackrel{\sim}{\longrightarrow}\im{({\renewcommand{\arraystretch}{0.6}\begin{array}{@{}c@{}}\scriptstyle A\\\scriptstyle B\end{array}})}$ inside the $7$-dimensional space $\ker(F_-,-F_+)$. Hence, the map
\[
K^2=U_\alpha \;\hookrightarrow\; \ker(F_-,-F_+) \;\surj\; \HM_\eta
\]
is an isomorphism.

\subsection{Optimisation}
\label{subsec:Optim}
The map $(F_-,-F_+)$ in the monad of generic stalks~\eqref{eq:MonadGS} is induced by the map $q\colon\Lambda^2\CT\oplus\Lambda^2\CT\surj V^*(-2)\otimes\CO(\partial\PP^4)$ in the monad of sheaves~\eqref{eq:Monad}. In view of Proposition~\ref{prop:KerWeilDecos}, the Weil decoration of $\ker q$ is the restriction of 
\[
\WD_{\Lambda^2\CT\oplus\Lambda^2\CT}(v^1,v^2)=\min\{
\WD_{\Lambda^2\CT}(v^1),\,\WD_{\Lambda^2\CT}(v^2)\},
\]
so we start with the Weil decoration of $\Lambda^2\CT$. The dual Euler sequence~\eqref{eq:DualEuler} induces
\begin{equation}
\label{eq:EulerLambda2}
\xymatrix{
\fbox{$\Lambda^2V_K(1)=(\Lambda^2V)\otimes_\kk\QCox_2$}\ar[r]^-{\iota^*_\eta}&\fbox{$\Lambda^2\CT_\eta=\Lambda^2N_K$}\ar[r]&0
}
\end{equation}
at generic stalk level. For the subsequent lemma and the remainder of this section we use the basis $[ij]=a_i\wedge a_j$ from~\eqref{eq:Lambda2NKBasis} for $\Lambda^2N_K$ and agree to let {\em latin indeces run from $1$ to $4$} and {\em greek indeces from $0$ to $4$}.

\begin{lemma}
\label{lem:WDL2T}
For $f_{ij}\in K$ with $f_{ji}=-f_{ij}$ we find
$\WD_{\Lambda^2\CT_\eta}\big(\sum_{i<j}f_{ij}[ij]\big)_\PD=\min\limits_{i<j}\{|f_{ij}|\}$ if $\PD\in\TT\subseteq\PP^4$, and  
\begin{equation}
\label{eq:WDLambda2T}
\WD_{\Lambda^2\CT_\eta}\big(\sum_{i<j}f_{ij}[ij]\big)_\PD=\min\limits_{j,\,\ell\neq k}\{\ord_{H_k}(f_{jk})+1,\ord_{H_k}(f_{j\ell})\}
\end{equation}
if $\PD=H_k$, $k\geq1$.
\end{lemma}

\begin{proof}
In order to apply Proposition~\ref{prop:QuotWeilDecos} to~\eqref{eq:EulerLambda2} we need to compute the fibre $\iota^{*-1}_\eta\big(\sum_{i<j}f_{ij}[ij]\,\big)$ first. Every element $\omega$ in $\Lambda^2V_K(1)=(\Lambda^2V)\otimes_\kk\QCox_2$ can be written as 
\begin{equation}
\label{eq:StalkEl}
\omega=\sum_{\mu<\nu}
\frac{\partial}{\partial z_\mu}\wedge\frac{\partial}{\partial z_\nu}\otimes \omega_{\mu\nu}z_\mu z_\nu=\sum_{\mu<\nu}\omega_{\mu\nu}\,\partial_\mu\wedge\partial_\nu
\end{equation}
for $\omega_{\mu\nu}\in K$ which implies
\[
\iota^*_\eta(\omega)=\sum_{i<j}(\omega_{ij}-\omega_{0j}+\omega_{0i})[ij].
\]
Indeed, $a_0=-\sum a_i$, and $\partial_\nu$ maps to $a_\nu$ by~\eqref{eq:EulerSurj} whence $\sum_j\frac{\partial}{\partial z_0}\wedge\frac{\partial}{\partial z_j}\otimes \omega_{0j}z_0z_j$ maps to $-\sum_{i,j}\omega_{0j}(a_i\wedge a_j)=\sum_{i<j}(\omega_{0i}-\omega_{0j})[ij]$. Now for given coefficients $f_{ij}\in K$, $\iota^*_\eta(\omega)=\sum_{i<j}f_{ij}[ij]$ for some $\omega$ requires $f_{ij}=\omega_{ij}-\omega_{0j}+\omega_{0i}$. Setting $\tau_k=\omega_{0\ell}$ entails
\[
\iota^{*-1}_\eta\big(\sum_{i<j}f_{ij}[ij]\,\big)=\Big\{\sum_\ell\tau_\ell\,\partial_0\wedge\partial_\ell+\sum_{i<j}
(f_{ij}+\tau_j-\tau_i)\,\partial_i\wedge\partial_j\mid\tau_1,\ldots,\tau_4\in K\Big\}\cong K^4.
\]
For any prime divisor $\PD$ in $\PP^4$, the Weil decoration of the direct sum of line bundles $(\Lambda^2V)(2)=\kk^{10}\otimes_k\CO(2)=\CO(2)^{10}$ is given by
\[
\WD_{\Lambda^2(V(1))}(\omega)_\PD=\WD_{(\Lambda^2V)(2)}(\omega)_\PD=\min\limits_{\mu<\nu}\!
\big\{\ord_\PD(\omega_{\mu\nu})+(H_\mu+H_\nu)_{\!\PD}\big\},
\]
cf.\ Corollary~\ref{coro:DirSum},~\eqref{eq:WDOl} and~\eqref{eq:StalkEl}. For any form $\varphi=\sum_{i<j}
\!f_{ij}[ij]\in\Lambda^2\CT$, Proposition~\ref{prop:QuotWeilDecos} implies
\begin{align*}
&\WD_{\Lambda^2\CT}(\varphi)_\PD=\max\limits_{\omega\in\iota^{*-1}_\eta(\varphi)}{\WD_{\Lambda^2(V(1))}(\omega)_\PD}\;=\\
&\max\limits_{\tau_\bullet\in K}\big\{\!\!\!\min\limits_{{\begin{array}{c}\scriptstyle i<j\\[-3pt]\scriptstyle\ell\end{array}}}\!\{\ord_\PD(\tau_\ell)+(H_0)_\PD+(H_\ell)_\PD,\,\ord_\PD(f_{ij}+\tau_j-\tau_i)+(H_i)_\PD+(H_j)_\PD\}\big\}.
\end{align*} 
If $\PD$ is in the torus, then the usual properties of valuations imply
\[
\min\{\ord_\PD(f_{ij}+\tau_j-\tau_i),\,\ord_\PD(\tau_i),\,\ord_\PD(\tau_j)\}=\min\{\ord_\PD(f_{ij}),\,\ord_\PD(\tau_i),\,\ord_\PD(\tau_j)\}
\]
for any pair of indeces $i<j$. This is the first case of Lemma~\ref{lem:WDL2T}. 

\smallskip

Next, let $\PD=H_k$, $k\neq0$; without loss of generality, take $\PD=H_1$ for sake of concreteness, and consider
\[
(*):=\min\limits_{1<\ell,\,1<i<j}\{\ord_\PD(\tau_1)+1,\,\ord_\PD(\tau_\ell),\,\ord_\PD(f_{1\ell}+\tau_\ell-\tau_1)+1,\,\ord_\PD(f_{ij}+\tau_j-\tau_i)\}
\]
for given $\tau_1,\ldots,\tau_4\in K$. Arguing as before we deduce $(*)\leq\ord_\PD(f_{1\ell})+1$ for $1<\ell$ and $(*)\leq\ord_\PD(f_{ij})$ for $1<i<j$. On the other hand, taking $\tau_1=\ldots=\tau_4=0$ yields
\begin{align*}
\max\{(*)\}\geq\min\{&\ord_\PD(f_{12})+1,\ord_\PD(f_{13})+1,\ord_\PD(f_{14})+1,\\
&\ord_\PD(f_{23}),\ord_\PD(f_{24}),\ord_\PD(f_{34})\}.  
\end{align*}
This gives~\eqref{eq:WDLambda2T}.
\end{proof}

\begin{remark}
Though we won't use it later on we note that after substituting $a_1$ by $-\sum_{\nu\neq1}a_\nu$, a straightforward computation yields the missing case
\[
\WD_{\Lambda^2\CT_\eta}\big(\sum f_{ij}[ij]\big)_{H_0}=\min\limits_{1<\ell,\,0<i<j<k}\{\ord_{H_0}(f_{1\ell})+1,\,\ord_{H_0}(f_{ij}-f_{ik}+f_{jk})\}. 
\]
\end{remark}

Next, consider the short exact sequence of sheaves $0\to V(2)\to\ker q\to\HM\to0$ with corresponding short exact sequence
\[
0\to V_K(2)\to\fbox{$\ker q_\eta=V_K(2)\oplus U_\alpha$}\to\fbox{$\HM_\eta=U_\alpha$}\to0
\]
of $K$-vector spaces. Surjectivity of $\ker(q)\subseteq\Lambda^2\CT\oplus\Lambda^2\CT\surj\HM$ implies 
\[
\WD_\HM(e)=\bigwedge\big\{\WD_{(\Lambda^2\CT\oplus\Lambda^2\CT)}(e + V_K(2))\big\}
\]
for each $e\in \HM_\eta=U_\alpha$ by Proposition~\ref{prop:QuotWeilDecos}. 
More explicitely, write any $e\in U_\alpha$ inside 
$\Lambda^2N_K\oplus\Lambda^2N_K$ as $e=(f\alpha^-\hspace{-0.3em},\newMinus g\alpha^+)$ for $(f,g)\in K^2\setminus\{(0,0)\}$, that is,
\[
e=f\cdot([12],0)\oldPlusWirdMinus g\cdot(0,[13])+f\cdot([14],0)
\oldPlusWirdMinus g\cdot(0,[23])\oldPlusWirdMinus g\cdot(0,[24])+f\cdot([34],0),
\]
cf.~\eqref{eq:AlphaPM}. By varying $\underline h=(h_0,\ldots,h_4)\in K^5$ we need to maximise the divisor
\begin{equation}
\label{eq:DHMForm}
\WD_\HM(f,g)(\underline h):=
\WD_{\Lambda^2\CT}\big(\widetilde{A}\cdot(\underline{h},f)^\top\big)\wedge\WD_{\Lambda^2\CT}\big(\widetilde{B}\cdot(\underline{h},g)^\top\big)\end{equation}
where $\widetilde{A}=(A|\alpha^-)$ and $\widetilde{B} = (B|\newMinus\alpha^+)$ are the $(6\times 6)$-matrices obtained from~\eqref{eq:MonadGS} and~\eqref{eq:AlphaPM}. To ease notation, we put 
\begin{equation}
\label{eq:MonFac}
a_\nu(\underline{z}):=\frac{z_\nu^2}{z_{\nu+2}\,z_{\nu-2}}
\quad\text{and}\quad
b_\nu(\underline{z}):=\frac{z_\nu^2}{z_{\nu-1}\,z_{\nu+1}}
\end{equation}
for $\nu\in\Z/5\Z$ and compute
\begin{align}
\widetilde{A}\cdot(\underline h,f)^\top=&\big((a_3h_3+a_4h_4+f)[12]+a_3h_3[13]+(a_2h_2+a_3h_3+f)[14]\nonumber\\
&\;+a_0h_0[23]+a_2h_2[24]+(a_1h_1+a_2h_2+f)[34]\big)^\top,\nonumber\\
\widetilde{B}\cdot(\underline h,\newMinus g)^\top=&
\big(b_1h_1[12]-(b_2h_2+b_4h_4\oldMinusWirdPlus g)[13]+b_0h_0[14]\nonumber\\
&\;-(b_1h_1+b_4h_4\oldMinusWirdPlus g)[23]-(b_1h_1+b_3h_3 \oldMinusWirdPlus g)[24]+b_4h_4[34]\big)\label{eq:VecEval}.
\end{align}
Fix a prime divisor $\PD$ and let $|\cdot|$ be shorthand for $\ord_\PD$. We set out to prove the formula of Theorem~\ref{thm:WDofHM}\,, that is, for $\PD=H_\nu=\{z_\nu=0\}$, $\nu\in\Z/5\Z$, we have
\[
\WD_\HM(f,g)_{H_\nu}=\min\{\ord_{H_\nu}(f),\ord_{H_\nu}(g)\}+1\quad
\text{if }(f/g)(H_\nu)=\oldMinus\frac{z_{\nu+1}z_{\nu-1}}{z_{\nu+2}z_{\nu-2}}
\]
and $\WD_\HM(f,g)_\PD=\min\{\ord_\PD(f),\ord_\PD(g)\}$ for $\PD$ else.

\medskip

The case $\PD\in\TT$ is a routine check left to the reader. For $\PD=H_k$, $k\geq0$, we assume $k=1$, as the setup is clearly symmetric in the boundary divisors $H_k$. First, we can discard straightaway the $h_0$-term when taking the maximum. Since $\underline{h}=0$ yields the lower bound $\WD_\HM(f,g)_{H_1}\geq\fbox{$m:=\min\{|f|,|g|\}$}$ from~\eqref{eq:DHMForm} we analyse what happens for $\WD_\HM(f,g)_{H_1}>m$. The monomial factors $a_\nu$ and $b_\nu$ from Equation~\eqref{eq:MonFac} contribute
\begin{equation}
\label{eq:MonimalsHi}
|a_1|=|b_1|=2,\quad|a_3|=|a_4|=|b_0|=|b_2|=-1\quad\text{and}\quad0\text{ otherwise} 
\end{equation}
with respect to $|\cdot|=\ord_{H_1}$. From Lemma~\ref{lem:WDL2T} and Equation~\eqref{eq:VecEval} we gather
\begin{align*}
\renewcommand{\arraystretch}{1.3}
\WD_{\Lambda^2\CT}\big(\widetilde A\cdot(\underline{h},f)^\top\big)_{H_1}=&\min\{|a_3h_3+a_4h_4+f|+1,\,|h_3|,|a_2h_2+a_3h_3+f|+1,\\
&\;|h_2|,\,|a_1h_1+a_2h_2+f|\}.
\end{align*}
In particular, $|h|_2\geq m+1$ whence $|a_1h_1+f|\geq m+1$, for $|a_1h_1+a_2h_2+f|=\min\{|a_1h_1+f|,|a_2h_2|\}$. As a result, the terms of order $m$ in $a_1h_1+f$, if any, must cancel. Similarly, looking at $\WD_{\Lambda^2\CT}\big(\widetilde B\cdot(\underline{h},g)^\top\big)_{H_1}$, we see that $b_1h_1+g$ has no terms of order $m$. For a local parameter $t\in\idm_{\PP^4,H_1}$ both cancellations happen simultanously if and only if
\[
(a_1h_1t^{-m})(H_1)=-(ft^{-m})(H_1)\quad\text{and}\quad
(b_1h_1t^{-m})(H_1)=\newMinus (gt^{-m})(H_1), 
\]
or equivalently, if and only if $(b_1ft^{-m})(H_1)=\oldMinus(a_1gt^{-m})(H_1)$. This means that with respect to the degree in $t$, the lowest order term of $b_1ft^{-m}$ must equal the lowest order term of $\oldMinus a_1gt^{-m}$, or equivalently, 
\[
(f/g)(H_1)=\oldMinus a_1/b_1=\oldMinus\frac{z_{0}z_{2}}{z_{3}z_{4}}.
\]
This completes the proof of Theorem~\ref{thm:WDofHM}.

\section{Reflexive sheaves of \texorpdfstring{$\HM$}{HM}-type}
\label{sec:HMtype}
In the previous Section~\ref{sec:HMBundle} we derived the Weil decoration of the classical Horrocks-Mumford bundle. Next, we axiomatise and generalise this construction. For this, let $X$ an algebraic variety over $\kk$ with a simple normal crossing divisor $D$.

\begin{proposition}
\label{prop:WDHMtype}
For each prime divisor $D_\rho$ supporting $D$ choose a unit $h_\rho\in\kappa(D_\rho)^*$. The assignment $\WD_h\colon K^2\to\hatDiv(X)$ determined by
\[
\WD_h(f,g)_\PD:=
\begin{cases}
\min\{\ord_P(f),\ord_P(g)\}+1,\!\!&\PD=D_\rho\text{ and }
(f/g)(D_\rho)=h_\rho\\
\min\{\ord_P(f),\ord_P(g)\},\!\!&\text{else}
\end{cases}
\]
defines a Weil decoration.
\end{proposition}

\begin{proof}
Setting $h_{D_\rho}=h_\rho$ in Proposition~\ref{prop:SemiNorm}\,, $\WD_h$ defines a pre-Weil decoration. Moreover, $\WD_h$ differs from a trivial Weil decoration by at most $D$. Proposition~\ref{prop:CohCond} implies that $\WD$ is a Weil decoration.
\end{proof}

\begin{definition}
\label{def:HMtype}
The reflexive rank two sheaf associated with $\WD_h$ is called a {\em H(orrocks)-M(umford) sheaf on $(X,D)$}, written $\HM(X,D,h)$ or simply $\HM(h)$.
\end{definition}

\begin{example}
The Horrocks-Mumford bundle $\HM$ on $\PP^4$ is obtained by taking $D=\partial\PP^4=\sum_{\rho=0}^4H_\rho$ for the coordinate hyperplanes $H_\rho=\{z_\rho=0\}$, and
\begin{equation}
\label{eq:HMh}
h_\rho=z_{\rho+1}z_{\rho-1}/z_{\rho+2}z_{\rho-2}\in\kappa^*(H_\rho),\quad\rho\in\Z/5\Z. 
\end{equation}
\end{example}

\begin{remark}
\label{rem:CanInc}
By design, we have $\CO_X^2\subseteq\HM_u(X,D)\subseteq\CO_X(D)^2$. In particular, $e_1=(1,0)$ and $e_2=(0,1)$ provide a $K$-basis of $\HM_h(X,D)_\eta=K^2$ which in general is not $X$-orthogonal, but satisfies $(\WD_h)_E\equiv0$ for its induced slice $E=\spann_\kk\{e_1,\,e_2\}$; cf.\ also Remark~\ref{rem:Slices} (ii). 
\end{remark}

The previous example highlights a natural subclass of $\HM$-sheaves given by toric varieties $X=\toric(\Sigma)$ and their anti-canonical divisor $D=\partial X=\sum_{\rho\in\Sigma(1)}D_\rho$. Recall that $x^m\in\kk[M]$ is the regular function on the torus defined by $m$ in the character lattice $M$, cf.\ \ssect{subsec:TS}. For a ray $\rho$, pick $u_\rho\in M$ such that the orthogonality condition $\langle u_\rho,\rho\rangle=0$ holds; this ensures that 
$x^{u_\rho}\in\CO_X(\TT)\subseteq K$  is a unit in the local ring $\CO_{X,D_\rho}$, that is, $h_\rho=x^{u_\rho}(D_\rho)\in\kappa(D_\rho)^*$. By~\eqref{eq:FundSeq}, we actually may view $u_\kbb$ as a map
\[
\Div_\TT(X)=\Z^{\Sigma(1)}\to M\subseteq\Div_\TT(X) 
\]
conveniently encoded in an integer $\sharp\Sigma(1)\times\sharp\Sigma(1)$-matrix with vanishing diagonal and $[\cdot]\circ u=0$, where $[D]\in\opn{Cl}(X)$ is the class of the divisor $D$. For instance, the rows of $u$ must add to zero for $X=\PP^n$. The restriction does not apply if $\opn{Cl}(X)=0$, e.g., $X=\A^n_\kk$. We write \fbox{$\HM(X,u)$} or simply $\HM(u)$ for the corresponding $\HM$-sheaf.

\begin{example}
For the classical Horrocks-Mumford bundle given by~\eqref{eq:HMh}, we find
\begin{equation}
\label{eq:MatHM}
u=\Matr{5}{ 
0&1&-1&-1&1\\
1&0&1&-1&-1\\
-1&1&0&1&-1\\
-1&-1&1&0&1\\
1&-1&-1&1&0};
\end{equation}
as an additional feature, the matrix is symmetric.
\end{example}

\begin{proposition}
For a smooth semi-projective toric variety $X^n$ with $u=0$, we have an explicit isomorphism
\[
\HM(X,0)=\CO_X\oplus\CO_X(\partial X).  
\]
\end{proposition}

\begin{proof}
Consider the complex 
\[
\xymatrix{
0\ar[r]&\HM(X,0)\ar[r]&\CO_X(\partial X)^2\ar[r]^-\Psi&\CO_X\big(\partial X\big)\Big/\CO_X\ar[r]&0 
}
\]
induced by $\Psi(f,g)=\overline{f-g}$, the residue class of $f\oldPlusWirdMinus g$. Since $X$ is smooth and semi-projective, the torus-invariant open affines $U_\sigma=\Spec\kk[x_\rho\,|\,\rho\in\sigma(1)]=\A^n_\kk$ associated with the top-dimensional cones $\sigma\in\Sigma(n)$ cover $X$. The local bases $e_0=(1,1)/\prod_{\rho\in\sigma(1)}\!x_\rho$ and $e_1=(1,0)$ then glue to the diagonal $\Delta\big(\CO_X(\partial X)\big)$ and $\CO_X\oplus0$ inside $\CO_X(\partial X)^2$.
\end{proof}

\begin{example}
For $u\neq0$, the complexity increases quickly as revealed by the following computations (which were carried out by Oscar~\cite{OSCAR}). For instance, already in dimension $3$, there are non-locally free $\HM$-sheaves. Indeed, consider the matrices 
\[
u_1=\left(\begin{array}{rrr}
0&1&0\\
1&0&-1\\
0&-1&0
\end{array}\right)
\quad\text{and}\quad
u_2=\left(\begin{array}{rrr}
0&1&-1\\
1&0&-1\\
-1&-1&0
\end{array}\right).
\]
The $\kk[x_1,x_2,x_3]$-module $\HM(\A^3,u_1)$ is not free by Nakayama, for its generators
\[
v_1=\big(\frac{1}{x_2x_3}+\frac{1}{x_1},\;
\frac{1}{x_3}+\frac{1}{x_1x_2}\big),\;
v_2=\big(\frac{1}{x_2}+\frac{x_3}{x_1},\;
\oldMinus\frac{x_3}{x_1x_2}\big)\;\text{ and }\; 
v_3=\big(\frac{x_2}{x_1},\;\oldMinus\frac{1}{x_1}\big)
\]
satisfy the syzygy $x_3v_1+(x_2^2-1)v_2-(x_1+x_2x_3)v_3=0$ vanishing at
$(0,\pm1,0)$. On the other hand, the seemingly more complicated
$\HM(\A^3,u_2)$ has minimal number of (free) generators
\[
v_1=\big(\frac{1}{x_2x_3}+\frac{1}{x_1x_3},\;\frac{x_1}{x_3}+\frac{x_2}{x_3}+\frac{1}{x_1x_2}\big)
\;\text{ and }\; 
v_2=\big(\frac{x_2}{x_1x_3},\;\frac{x_2}{x_3}+\frac{1}{x_1}\big).
\]
\end{example}

\begin{remark}
In general, the $\kk[\ux]$-module $\HM(\A^n,u)$, which appears, for instance, as $\HM$-sheaf on an affine chart of $\PP^n$ or on the affine cone over $\PP^{n-1}$, equals the reflexive hull of
\[
\Big\langle\frac{1}{x_i}\cdot\big(x^{u^+(i)},\,
\oldMinus x^{u^-(i)}\big)\kSt
i=1,\ldots,n\Big\rangle+\kk[\ux]^2\;\subseteq\;\frac{1}{x_1\cdot\ldots\cdot
x_n}\cdot\kk[\ux]^2;
\]
here, $u^+(i)$ and $u^-(i)\in\N^n$ denote the positive and negative part of
$u_i=u^+(i)-u^-(i)$ in $M=\Z^n$. Alternatively, we can compute $\HM(\A^n,u)$ as the intersection
\[
\HM(u)\;=\;\frac{1}{x_1\cdot\ldots\cdot x_n} \cdot
\bigcap\limits_{i=1}^n\Big\langle (x^{u^+(i)},\,
\oldMinus x^{u^-(i)}),\;(x_i,0),\;(0,x_i)\Big\rangle.
\]
For details, see~\cite{WHMtoric}.
\end{remark}

\end{document}